\documentclass{siamltex1213}

\usepackage{amsmath,amssymb,amsfonts,cmmib57,mathdots,xcolor,mathrsfs}
\usepackage[latin1]{inputenc}
\usepackage{graphicx}
\usepackage{enumerate}

\newcommand{\rev}{\ensuremath\mathrm{rev}}

\newtheorem{remark}[theorem]{Remark}

\title{The conditioning of block Kronecker $\ell$-ifications of matrix polynomials}
\author{Javier P\'{e}rez\footnotemark[1]}
\begin{document}
\maketitle
\slugger{simax}{xxxx}{xx}{x}{x--x}
\renewcommand{\thefootnote}{\fnsymbol{footnote}}
\footnotetext[1]{Department of Mathematical Sciences, University of Montana, USA. Email: {\tt javier.perez-alvaro@mso.umt.edu}.}

\renewcommand{\thefootnote}{\arabic{footnote}}

\begin{abstract}
A strong $\ell$-ification of a matrix polynomial $P(\lambda)=\sum A_i\lambda^i$ of degree $d$ is a matrix polynomial $\mathcal{L}(\lambda)$ of degree $\ell$ having the same finite and infinite elementary divisors, and the same number of left and right minimal indices as $P(\lambda)$.
Strong $\ell$-ifications can be used  to transform the polynomial eigenvalue problem associated with $P(\lambda)$ into an equivalent polynomial eigenvalue problem associated with a larger matrix polynomial $\mathcal{L}(\lambda)$ of lower degree. 
Of most interest in applications is $\ell=1$, for which $\mathcal{L}(\lambda)$ receives the name of strong linearization. 
However, there exist some situations, e.g., the preservation of algebraic structures, in which it is more convenient to replace strong  linearizations by other low degree matrix polynomials.
In this work, we investigate the eigenvalue conditioning of $\ell$-ifications from a family of matrix polynomials recently identified and studied by Dopico, P\'erez and Van Dooren, the so-called block Kronecker companion forms.
We compare the conditioning of these $\ell$-ifications with that of the matrix polynomial $P(\lambda)$, and show that they are  about as well conditioned as the  polynomial itself, provided we scale $P(\lambda)$ so that $\max\{\|A_i\|_2\}=1$, and the quantity $\min\{\|A_0\|_2,\|A_d\|_2\}$ is not too small. 
Moreover, under the scaling assumption  $\max\{\|A_i\|_2\}=1$, we show that any block Kronecker companion form, regardless of its degree or block structure, is about as well-conditioned as the  well-known Frobenius companion forms.
Our theory is illustrated by numerical examples.
\end{abstract}

\pagestyle{myheadings}
\thispagestyle{plain}

\section{Introduction}\label{sec:intro}

Finding eigenvalues of matrix polynomials is an important task in scientific computation.
In the present paper, we study the conditioning of solving a polynomial eigenvalue problem by using $\ell$-ifications, and its implications on the accuracy of computed eigenvalues. 

An $n\times n$ \emph{matrix polynomial} takes the form
\begin{equation}\label{eq:poly}
P(\lambda) = \sum_{i=0}^d A_i\lambda^i, \quad \mbox{with} \quad A_0,A_1,\hdots,A_d\in\mathbb{C}^{n\times n}.
\end{equation}
We assume throughout that the polynomial $P(\lambda)$ is \emph{regular}, this is, the scalar polynomial $\det P(\lambda)$ is not identically equal to the zero polynomial. 
When $A_d\neq 0$, we say that $P(\lambda)$ has \emph{degree} $d$, otherwise we say that $P(\lambda)$ has \emph{grade} $d$.
The \emph{polynomial eigenvalue problem (PEP)} associated with $P(\lambda)$ consists in finding scalars $\lambda\in\mathbb{C}$ and nonzero vectors $x,y\in\mathbb{C}^n$ satisfying
\[
P(\lambda)x = 0 \quad \mbox{and} \quad y^* P(\lambda) = 0,
\]
where $(\cdot)^*$ denotes the complex conjugate transpose.
The scalar $\lambda$ is called an \emph{eigenvalue} of $P(\lambda)$, and the vectors $x$ and $y$ are, respectively, the \emph{right and left eigenvectors} of $P(\lambda)$ associated with the eigenvalue $\lambda$.
The pairs $(\lambda,x)$ and $(y,\lambda)$ are called, respectively, \emph{right and left eigenpairs} of $P(\lambda)$.
The triple $(y,\lambda,x$) is called an \emph{eigentriple} of $P(\lambda)$.
We assume that the reader has some familiarity with matrix polynomials and polynomial eigenvalue problems.
For those readers not familiar with these concepts, we refer to the classical works \cite{Gantmacher,Lancaster}, or to the recent reference \cite{Index Sum Theorem} and the references therein.

Any numerical algorithm for computing eigenvalues of matrix polynomials  is affected by roundoff errors due to the limitations of floating point arithmetic.
Ideally, we would like to use polynomial eigenvalue solvers that are at least forward stable, this is,  algorithms that are able to find well-conditioned eigenvalues with high relative accuracy
\begin{equation}\label{eq:forward stability}
\dfrac{|\lambda-\widetilde{\lambda}|}{|\lambda|}=\mathcal{O}(u)\kappa_P(\lambda) \quad \quad \lambda\mbox{: exact eigenvalue,} \quad  \quad \widetilde{\lambda}\mbox{: computed eigenvalue}.
\end{equation}
where $u$ is the unit roundoff, and we use the notation $\mathcal{O}(u)$ for any quantity that is upper bounded by $u$ times a modest constant.
In \eqref{eq:forward stability} $\kappa_P(\lambda)$ denotes the conditioning of the eigenvalue $\lambda$ of  the matrix polynomial $P(\lambda)$.
The condition number $\kappa_P(\lambda)$ measures how sensitive is the eigenvalue $\lambda$ to perturbations of the matrix coefficients of $P(\lambda)$ \cite{Tisseur}.
More precisely, if $\lambda$ is a simple, finite, nonzero eigenvalue of a matrix polynomial $P(\lambda)$ as in \eqref{eq:poly}, with corresponding right and left eigenvectors $x$ and $y$,  then the \emph{condition number of $\lambda$} is defined by 
\begin{align*}
\kappa_P(\lambda) := \lim_{\epsilon \rightarrow 0} \sup 
\left\{
\frac{|\Delta \lambda|}{\epsilon |\lambda|} \, : \, \left( \sum_{i=0}^d (A_i+\Delta A_i)(\lambda+\Delta \lambda)^i \right)(x+\Delta x)=0,\right. &\\
\left.\phantom{\frac{|\Delta \lambda|}{\epsilon |\lambda|}} \mbox{with} \, \, \|\Delta A_i\|_2 \leq \epsilon \omega_i
\right\}&,
\end{align*}
where the tolerances $\omega_i$ provide some freedom in how perturbations are measured.
An explicit formula for $\kappa_P(\lambda)$ was given by Tisseur (see \cite[Theorem 5]{Tisseur}):
\begin{equation}\label{eq:condition number}
\kappa_P(\lambda) = \frac{\left(\sum_{i=0}^d |\lambda|^i \omega_i\right)\|x\|_2\|y\|_2}{|\lambda|\cdot |y^*P^\prime(\lambda)x|}.
\end{equation}
Then, we say that a simple, finite, nonzero eigenvalue $\lambda$ of a matrix polynomial $P(\lambda)$ is \emph{well-conditioned} when $\kappa_P(\lambda)\approx 1$.

Of most interest are the choices $\omega_i=\|A\|_2$, for $i=0,1,\hdots,d$, for which we call $\kappa_P(\lambda)$ the \emph{coefficientwise condition number}, and  $\omega_i=\|\begin{bmatrix}
A_0 & A_1 & \cdots & A_d
\end{bmatrix}\|_2$, for $i=0,1,\hdots,d$, for which we call $\kappa_P(\lambda)$ the \emph{normwise condition number}.
We denote the coefficientwise condition number of $\lambda$ by
\[
\mathrm{coeff\,cond}_P(\lambda) = \frac{\left(\sum_{i=0}^d |\lambda|^i \|A_i\|_2 \right)\|x\|_2\|y\|_2}{|\lambda|\cdot |y^*P^\prime(\lambda)x|},
\]
and the normwise condition number of $\lambda$ by 
\[
\mathrm{norm\,cond}_P(\lambda) = \frac{\|\begin{bmatrix} A_0 & \cdots & A_d \end{bmatrix}\|_2 \left(\sum_{i=0}^d|\lambda|^i\right)\|x\|_2\|y\|_2}{|\lambda|\cdot |y^*P^\prime (\lambda) x|}.
\]
Readily from their definitions, it follows $\mathrm{norm\,cond}_P(\lambda)\geq \mathrm{coeff\,cond}_P(\lambda)$.
So, a well-conditioned eigenvalue in the normwise sense is also well-conditioned in the coefficientwise sense, but not the other way around.

A common approach to solving a polynomial eigenvalue problem associated with a matrix polynomial $P(\lambda)$ starts by transforming $P(\lambda)$ into a larger matrix polynomial of  lower degree:
\[
\mathcal{L}(\lambda) = \sum_{i=0}^\ell \mathcal{L}_i\lambda^i \quad \mbox{with} \quad \mathcal{L}_0,\mathcal{L}_1,\hdots,\mathcal{L}_\ell\in\mathbb{C}^{m\times m}
\]
(see, for example, \cite{Robol,Index Sum Theorem,FFP,ell-ifications-DPV,Melman}).
The matrix polynomial $\mathcal{L}(\lambda)$ receives the name of $\ell$-ification of $P(\lambda)$ \cite{Index Sum Theorem}.
More specifically, an \emph{$\ell$-ification} of $P(\lambda)$ is a matrix polynomial $\mathcal{L}(\lambda)$ of degree at most $\ell$ such that 
\[
U(\lambda) \mathcal{L}(\lambda)V(\lambda) = \begin{bmatrix}
P(\lambda) & 0 \\ 0 & I_s
\end{bmatrix},
\]
for some unimodular matrices $U(\lambda)$ and $V(\lambda)$ \footnote{A matrix polynomial is unimodular if its determinant is a nonzero constant (independent of $\lambda$).}, and where  $I_s$  denotes the $s\times s$ identity matrix. 
This definition implies that $P(\lambda)$ and $\mathcal{L}(\lambda)$ have the same finite elementary divisor and, thus, the same finite eigenvalues with the same multiplicities \cite{Index Sum Theorem}. 

The matrix polynomial obtained by reversing the order of the matrix coefficients of $P(\lambda)$, i.e.,
\[
\rev P(\lambda):=\lambda^d P\left(\frac{1}{\lambda}\right)=\sum_{i=0}^d A_{d-i}\lambda^i,
\]
receives the name of the \emph{reversal matrix polynomial} of $P(\lambda)$.
If an $\ell$-ification $\mathcal{L}(\lambda)$ of $P(\lambda)$ satisfies additionally
\[
W(\lambda) \rev \mathcal{L}(\lambda)Z(\lambda) = \begin{bmatrix}
\rev P(\lambda) & 0 \\ 0 & I_s
\end{bmatrix},
\]
for some unimodular matrices $W(\lambda)$ and $Z(\lambda)$, the matrix polynomial $\mathcal{L}(\lambda)$ is called a \emph{strong $\ell$-ification} of $P(\lambda)$. 
The definition of strong $\ell$-ification implies that $P(\lambda)$ and $\mathcal{L}(\lambda)$ have the same finite and infinite elementary divisor and, thus, the same finite and infinite eigenvalues with the same multiplicities \cite{Index Sum Theorem}.

In numerical computations, the most common $\ell$-ifications used are those with $\ell=1$ \cite{NLEIGS,CORK AAA,CORK}.
When $\ell=1$, (strong) $\ell$-ifications receive the name of \emph{(strong) linearizations} \cite{Strong_lin}.
However, there are situations in which it is more convenient to transform $P(\lambda)$ into a matrix polynomial of degree larger than 1. 
For instance, it is known that there exist matrix polynomials with some important algebraic structures for which there are not strong linearizations preserving those structures \cite{Index Sum Theorem}. 
Since the preservation of algebraic structures has been recognized as a key factor for obtaining better and physically more meaningful numerical results \cite{GoodVibrations}, linearizations have been replaced by other low-degree matrix polynomials in some numerical computations.
For example, in \cite{pal_quadratization}, $2$-ifications (better known as quadratifications) are used in combination with doubling algorithms for solving even-degree structured PEPs.

Thegoal of this paper is to analyze the influence of the $\ell$-ification process on the eigenvalue conditioning, and on the accuracy of computed eigenvalues.
The key point is that the eigenvalues of $P(\lambda)$ are computed usually by applying to an $\ell$-ification $\mathcal{L}(\lambda)$ a backward stable algorithm---like the QZ algorithm in the case of linearizations, or doubling algorithms in the case of quadratifications\footnote{No backward stability proof exists so far for doubling algorithms, but in practice they produce small backward errors.}.
So, the computed eigenvalues are the exact eigenvalues of
\[
\mathcal{L}(\lambda)+\Delta \mathcal{L}(\lambda)=\sum_{i=0}^\ell(\mathcal{L}_i+\Delta\mathcal{L}_i)\lambda^i \quad \quad \mbox{with} \quad \quad \|\Delta\mathcal{L}_i\|_2\leq \mathcal{O}(u)\|\mathcal{L}_i\|_2.
\]
As a consequence of the backward stability, the forward error of a computed eigenvalue can be bounded as
\begin{equation}\label{eq:forward error}
\dfrac{|\lambda-\widetilde{\lambda}|}{|\lambda|} \leq \mathcal{O}(u)\,\mathrm{coeff\,cond}_\mathcal{L}(\lambda),
\end{equation}
where we recall that $\mathrm{coeff\,cond}_\mathcal{L}(\lambda)$ denotes the coefficientwise condition number of $\lambda$ as an eigenvalue of the $\ell$-ification $\mathcal{L}(\lambda)$.
However, in view of \eqref{eq:forward stability}, it is natural to look for a bound of the relative error $|\lambda-\widetilde{\lambda}|/|\lambda|$ in terms of the conditioning of the original polynomial $P(\lambda)$.
 This bound can be obtained by rewriting \eqref{eq:forward error} as
 \[
\dfrac{|\lambda-\widetilde{\lambda}|}{|\lambda|} \leq \mathcal{O}(u)\kappa_\mathcal{P}(\lambda) \cdot \dfrac{\mathrm{coeff\,cond}_\mathcal{L}(\lambda)}{\kappa_P(\lambda)}.
\]
Hence, the ratio $\mathrm{coeff\,cond}_\mathcal{L}(\lambda)/\kappa_P(\lambda)$ controls how far the eigensolver based on $\ell$-ification may be from being forward stable.
In view of this, one should use $\ell$-ifications such that $\mathrm{coeff\,cond}_\mathcal{L}(\lambda)\approx \kappa_P(\lambda)$, since, in this ideal situation, one would be able to compute well-conditioned eigenvalues with high relative accuracy.

In this work, we will focus both on  coefficient and normwise condition numbers. 
Specifically, we will study  the ratios
\[
\frac{\kappa_\mathcal{L}(\lambda)}{\kappa_P(\lambda)}=\frac{\mathrm{coeff\,cond}_\mathcal{L}(\lambda)}{\mathrm{coeff\,cond}_P(\lambda)} \quad \quad \mbox{and} \quad \quad \frac{\kappa_\mathcal{L}(\lambda)}{\kappa_P(\lambda)}=\frac{\mathrm{coeff\,cond}_\mathcal{L}(\lambda)}{\mathrm{norm\,cond}_P(\lambda)},
\]
when the $\ell$-ification $\mathcal{L}(\lambda)$ belongs to the family of block Kronecker companion forms (introduced in Section \ref{sec:companion forms}).
We recall that this family includes the  very well-known Frobenius companion forms \eqref{eq:C1} and \eqref{eq:C2}, permuted versions of the famous Fiedler pencils, and generalized Fiedler pencils \cite{vectors,simplified,DTDM10}.
Assuming $P(\lambda)$ is scaled so that $\max_{i=0:d}\{\|A_i\|_2\}=1$, we will show that there exist modest constants $c_1$ and $c_2$ such that
\begin{equation}\label{eq:main1}
\frac{\mathrm{coeff\,cond}_\mathcal{L}(\lambda)}{\mathrm{norm\,cond}_P(\lambda)}\leq c_1 \quad \quad \mbox{and} \quad \quad \frac{\mathrm{coeff\,cond}_\mathcal{L}(\lambda)}{\mathrm{coeff\,cond}_P(\lambda)}\leq \dfrac{c_2}{\min\{\|A_0\|_2,\|A_d\|_2 \}}.
\end{equation}
Additionally, if $\mathcal{R}(\lambda)$ is another block Kronecker companion form, then we will show that there exist modest constants $c_3$ and $c_4$ such that
\begin{equation}\label{eq:main2}
c_3 \leq \frac{\mathrm{coeff\,cond}_\mathcal{R}(\lambda)}{\mathrm{coeff\,cond}_\mathcal{L}(\lambda)} \leq c_4.
\end{equation}
These results are stated in Theorems \ref{thm:main1} and \ref{thm:main2}, which are the main contributions of this work.

Notice that \eqref{eq:main1} implies that well-conditioned eigenvalues in the normwise sense can be computed with high relative accuracy as the eigenvalues of any block Kronecker companion form if we apply a backward stable algorithm to the $\ell$-ification. 
A similar conclusion holds for well-conditioned eigenvalues in the coefficientwise sense, provided that  $\min\{\|A_0\|_2,\|A_d\|_2 \}$ is not too small. 
Moreover, we want to emphasized that \eqref{eq:main2} covers the case when $\mathcal{L}(\lambda)$ is one of the well-known Frobenius companion forms:
\begin{equation}\label{eq:C1}
C_1(\lambda):=\begin{bmatrix}
\lambda A_d+A_{d-1} & A_{d-2} & \cdots & A_1 & A_0 \\
-I_n & \lambda I_n & 0 & \cdots & 0 \\
0 & \ddots & \ddots & \ddots & \vdots \\
\vdots & \ddots & -I_n & \lambda I_n & 0 \\
0 & \cdots & 0 & -I_n & \lambda I_n
\end{bmatrix}
\end{equation}
and
\begin{equation}\label{eq:C2}
C_2(\lambda):=\begin{bmatrix}
\lambda A_d+A_{d-1} & -I_n & 0 & \cdots & 0 \\
A_{d-2} & \lambda I_n & \ddots & \ddots & \vdots \\
\vdots & 0 & \ddots & -I_n & 0 \\
A_1 & \vdots & \ddots & \lambda I_n & -I_n \\
A_0 & 0 & \cdots & 0 & \lambda I_n
\end{bmatrix},
\end{equation}
since $C_1(\lambda)$ and $C_2(\lambda)$ are particular instances of block Kronecker companion forms.
This is quite a surprising result, since \eqref{eq:main2} implies that any block Kronecker companion form, regardless of its degree or block structure, is about as well-conditioned as the Frobenius companion forms  \eqref{eq:C1} and \eqref{eq:C2}.
Hence, block Kronecker companion forms can be used
in the polynomial eigenvalue problem with  similar reliability than Frobenius companion forms.

We begin in Section \ref{sec:Kronecker} by introducing the families of block Kronecker matrix polynomials, block Kronecker $\ell$-ifications (Section \ref{sec:Kronecker ell-ifications}) and block Kronecker companion forms (Section \ref{sec:companion forms}).
In Section \ref{sec:one sided factorization} we establish one-sided factorizations and obtain eigenvector formulas needed for studying effectively the conditioning of block Kronecker $\ell$-ifications.
Section \ref{sec:conditioning1} studies the conditioning of block Kronecker $\ell$-ifications relative to that of the matrix polynomial.
Then, in Section \ref{sec:conditioning2}, we prove that block Kronecker companion forms are about as well conditioned as the original polynomial, provided we scale $P(\lambda)$ so that $\max_{i=1:d}\{\|A_i\|_2\}=1$, and the quantity $\min\{\|A_0\|_2,\|A_d\|_2\}$ is not too small. 
In Section \ref{sec:conditioning3}, we show that under the scaling assumption  $\max_{i=1:d}\{\|A_i\|_2\}=1$  no block Kronecker companion form is better or worse conditioned than any other block Kronecker companion form.
Finally, we present  in Section \ref{sec:experiments} extensive numerical experiments that support our theoretical results.

Throughout the paper, we use the following notation.
We denote by $\mathbb{C}[\lambda]$ the ring of polynomials in the variable $\lambda$ with complex coefficients.
 The set of $m \times n$ matrix polynomials, this is, the set of $m\times n$ matrices with entries in $\mathbb{C}[\lambda]$, is denoted by $\mathbb{C}[\lambda]^{m\times n}$.
 We denote by $I_n$ the $n\times n$ identity matrix, and by $0$ the matrix with all its entries equal to zero, whose size should be clear from the context. 
By $A\otimes B$, we denote the Kronecker product of the matrices $A$ and $B$.

The next lemma will be useful when taking norms of block matrices (see \cite[Lemma 3.5]{BackErrors} and \cite[Proposition 3.1]{tridiagonal}).
\begin{lemma} \label{lemma:norm bound}
For any $p\times q$ block matrix $A=[A_{ij}]$, we have $\max_{ij}\{\|A_{ij}\|_2\}\leq\|A\|_2\leq \sqrt{pq}\max_{ij}\{ \|A_{ij}\|_2\}$.
\end{lemma}

\section{Block Kronecker matrix polynomials}
\label{sec:Kronecker}
A PEP can be transformed into an equivalent PEP associated with a larger matrix polynomial of lower degree (typically of degree 1 or 2) by using strong $\ell$-ifications \cite{Index Sum Theorem}.
Several approaches to constructing strong $\ell$-ifications have been introduced in the last years \cite{Robol,Index Sum Theorem,FFP,ell-ifications-DPV,Melman}.
In this work,  we focus on the strong $\ell$-ifications that belong to the family of block Kronecker matrix polynomials \cite{ell-ifications-DPV}.

We recall the family of block Kronecker matrix polynomials in Definition \ref{def:Kronecker-poly}. 
But, first, we introduce the following two matrix polynomials which are important for the definition and properties of this family.
\begin{equation}\label{eq:L}
L_k(\lambda) := \begin{bmatrix}
-1 & \lambda  & 0 & \cdots & 0 \\
0 & -1 & \lambda & \ddots & \vdots \\
\vdots & \ddots & \ddots & \ddots & 0\\
0& \cdots & 0 & -1 & \lambda
\end{bmatrix} \in\mathbb{C}[\lambda]^{k\times (k+1)}
\end{equation}
and
\begin{equation}\label{eq:Lambda}
\Lambda_k(\lambda) := \begin{bmatrix}
\lambda^k & \lambda^{k-1} & \cdots & \lambda & 1 
\end{bmatrix}^T \in\mathbb{C}[\lambda]^{(k+1)\times 1}.
\end{equation}
We observe that $L_k(\lambda) \Lambda_k(\lambda) = 0$, as this will be important in future sections.

\begin{definition}[Block Kronecker matrix polynomials]\label{def:Kronecker-poly}
An \emph{$(\epsilon,n,\eta,m)$-block Kronecker degree-$\ell$ matrix polynomial}, or simply a \emph{block Kronecker matrix polynomial}, is a degree-$\ell$ matrix polynomial  of the form
\[
\mathcal{L}(\lambda)=\left[ \begin{array}{c|c}
M(\lambda) & L_\eta(\lambda^\ell)^T\otimes I_m \\ \hline
L_\epsilon(\lambda^\ell)\otimes I_n & \phantom{\Big{(}} 0 \phantom{\Big{)}}
\end{array}\right],
\]
where $M(\lambda)\in\mathbb{C}[\lambda]^{(\eta+1)n\times(\epsilon+1)n}$ is an arbitrary grade-$\ell$ matrix polynomial, and where $L_k(\lambda)$ and $\Lambda_k(\lambda)$ are defined, respectively, in \eqref{eq:L} and \eqref{eq:Lambda}.
\end{definition}

A fundamental property of a block Kronecker matrix polynomial is that it is a strong $\ell$-ification of a certain $m\times n$ matrix polynomial.

\begin{theorem}{\rm \cite[Theorem 5.5]{ell-ifications-DPV}}\label{thm:block Kronecker poly}
The block Kronecker matrix polynomial \eqref{eq:block Kronecker poly} is a strong $\ell$-ification of the $m\times n$ matrix polynomial
\begin{equation}\label{eq:Q}
Q(\lambda):=(\Lambda_\eta(\lambda^\ell)^T\otimes I_m)M(\lambda)(\Lambda_\epsilon(\lambda^\ell)\otimes I_n),
\end{equation}
considered as a matrix polynomial of grade $\ell(\epsilon+\eta+1)$.
\end{theorem}

In practice, the matrix polynomial $Q(\lambda)$ in the left-hand-side of \eqref{eq:Q}  is given, and one would like to find a matrix polynomial $M(\lambda)$ satisfying \eqref{eq:Q}.
This inverse problem has been addressed in \cite{ell-ifications-DPV}.
We summarize the main results in the following section, focusing on the square case, that is, the case when $m=n$.

\subsection{Block Kronecker $\ell$-ifications for a prescribed matrix polynomial}\label{sec:Kronecker ell-ifications}

In this section, we are given an $n\times n$ matrix polynomial $P(\lambda)$ as in \eqref{eq:poly} of degree $d$, and a nonzero natural number $\ell$.
We assume that $d$ is divisible by $\ell$ (notice that this assumption is automatically satisfied when $\ell=1$).
Our goal is to construct strong $\ell$-ifications for $P(\lambda)$ by using block Kronecker matrix polynomials.

The starting point of the construction is to write $k:=d/\ell$ as $k=\epsilon+\eta+1$, for some nonnegative integers $\epsilon$ and $\eta$.
Then, from Theorem \ref{thm:block Kronecker poly}, we see that to get a strong $\ell$-ification for $P(\lambda)$ from a $(\epsilon,n,\eta,n)$-block Kronecker matrix polynomial, we must solve
\begin{equation}\label{eq:P}
(\Lambda_\eta(\lambda^\ell)^T\otimes I_n)M(\lambda)(\Lambda_\epsilon(\lambda^\ell)\otimes I_n) = P(\lambda),
\end{equation}
for the matrix polynomial $M(\lambda)$ of grade $\ell$.
Solving \eqref{eq:P} is always possible because this equation is consistent for every $n\times n$ matrix polynomial $P(\lambda)$.
The consistency of \eqref{eq:P} can be easily established as follows.
Based on the coefficients of $P(\lambda)$, we introduce the following grade-$\ell$ matrix polynomials
\begin{align}\label{eq:B_poly}
\begin{split}
&B_1(\lambda) :=  A_\ell \lambda^\ell +A_{\ell-1}\lambda^{\ell-1}+\cdots +  A_1\lambda + A_0, \\
&B_j(\lambda) :=  A_{\ell j}\lambda^\ell+A_{\ell j-1}\lambda^{\ell-1}+\cdots +  A_{\ell(j-1)+1}\lambda,\quad \mbox{for }j=2,\hdots,k.
\end{split}
\end{align}
Notice that if $P(\lambda)$ has degree $d$, i.e., $A_d\neq 0$, then $B_k(\lambda)$ has degree $\ell$.
Moreover, the polynomials $B_i(\lambda)$ satisfy the equality
\begin{equation}\label{eq:relation_B}
P(\lambda) = \lambda^{k(\ell-1)}B_k(\lambda)+\lambda^{k(\ell-2)}B_{k-1}(\lambda)+\cdots + \lambda^\ell B_2(\lambda)+B_1(\lambda).
\end{equation}

Then, from \eqref{eq:relation_B}, we can easily verify that the matrix polynomial
\begin{equation}\label{eq:M0}
M_{\epsilon,\eta}(\lambda;P):=
\begin{bmatrix}
B_k(\lambda) & B_{k-1}(\lambda) & \cdots & B_{\eta+1}(\lambda)\\
0 & \cdots & 0 & \vdots \\
\vdots & \ddots & \vdots & B_2(\lambda) \\
0 & \cdots & 0 & B_1(\lambda)
\end{bmatrix}
\end{equation}
is a solution of \eqref{eq:P}.
Hence, we have established the consistency of \eqref{eq:P}.
Observe that $M_{\epsilon,\eta}(\lambda;P)$ has degree $\ell$ when $P(\lambda)$ has degree $d$.

The matrix polynomial $M_{\epsilon,\eta}(\lambda;P)$  is not the only solution of \eqref{eq:P}; see \cite[Theorem 5.9]{ell-ifications-DPV}.
We recall in Theorem \ref{thm:characterization of solutions} two characterizations of the set of degree-$\ell$ solutions of \eqref{eq:P}.
Part (ii) in Theorem \ref{thm:characterization of solutions} gives a close formula for any solution of \eqref{eq:P}, while part (iii) shows how the coefficients of $P(\lambda)$ are distributed along the block anti-diagonals of any solution of \eqref{eq:P}.
\begin{theorem}\label{thm:characterization of solutions}
Let $P(\lambda)$ as in \eqref{eq:poly} be an $n\times n$ matrix polynomial of degree $d$. 
Assume $d$ is divisible by $\ell$, and set $k=d/\ell$.
Let $\epsilon$ and $\eta$ be nonnegative integers such that $\epsilon+\eta+1=k$. 
Then, the following conditions are equivalent.
\begin{itemize}
\item[\rm (i)] The degree-$\ell$ matrix polynomial $M(\lambda)$ satisfies \eqref{eq:P}.
\item[\rm (ii)] The matrix polynomial $M(\lambda)$ is of the form 
\begin{align}\label{eq:charac_BK}
\begin{split}
M(\lambda) = M_{\epsilon,\eta}(\lambda;P)+
&\left( \lambda \begin{bmatrix} 0 \\ D(\lambda) \end{bmatrix}+B \right)\left(L_\epsilon(\lambda^\ell)\otimes I_n\right)\\ &\quad\quad  +\left(L_\eta(\lambda^\ell)^T \otimes I_n\right)
\left( \lambda \begin{bmatrix} 0 & -D(\lambda) \end{bmatrix}+C\right),
\end{split}
\end{align}
for some matrices $B\in\mathbb{C}^{(\eta+1)n\times \epsilon n}$ and $C\in\mathbb{C}^{\eta n\times (\epsilon+1)n}$ and some matrix polynomial $D(\lambda)\in\mathbb{C}[\lambda]^{\eta n\times \epsilon n}$ of grade $\ell-2$, and where $M_{\epsilon,\eta}(\lambda;P)$ has been defined in \eqref{eq:M0}.
\item[\rm (iii)] If we consider  $M(\lambda)=\sum_{t=0}^\ell M_t\lambda^t$ as an $(\eta+1)\times( \epsilon+1)$ block matrix polynomial with $n\times n$ block entries, denoted by $[M(\lambda)]_{ij} = \sum_{t=0}^\ell [M_t]_{ij}\lambda^t$, then the matrix polynomial $M(\lambda)$ satisfies
 \begin{equation}\label{eq:condition_coeff1}
   \sum_{i+j=t+1} [M_\ell]_{ij} + \sum_{i+j=t} [M_0]_{ij}=A_{d-t\ell}, \quad \mbox{for $t=0,1,\hdots,k$,}
 \end{equation}
 and
 \begin{equation}\label{eq:condition_coeff2}
 \sum_{i+j=s+1}[M_{\ell-t}]_{ij} = A_{d-s\ell-t} \quad \mbox{for $s=0,1,\hdots,k-1$, and $t=1,\hdots,\ell-1$.}
 \end{equation}
\end{itemize}
\end{theorem}

Combining Theorems \ref{thm:block Kronecker poly} and \ref{thm:characterization of solutions} allows us to obtain infinitely many strong $\ell$-ifications of the prescribed matrix polynomial $P(\lambda)$.
This motivates the following definition.
\begin{definition}[Block Kronecker $\ell$-ification of $P(\lambda)$]
Given an $n\times n$ matrix polynomial $P(\lambda)$ as in \eqref{eq:poly} with degree $d$, we will refer to any block Kronecker matrix polynomial
\begin{equation}\label{eq:block Kronecker poly}
\mathcal{L}(\lambda)=\left[ \begin{array}{c|c}
M(\lambda) & L_\eta(\lambda^\ell)^T\otimes I_n \\ \hline
L_\epsilon(\lambda^\ell)\otimes I_n & \phantom{\Big{(}} 0 \phantom{\Big{)}}
\end{array}\right] \quad \quad \mbox{with} \quad \quad M(\lambda)=\sum_{i=0}^\ell M_i\lambda^i,
\end{equation}
where $M(\lambda)$  satisfies \eqref{eq:P}, as a \emph{block Kronecker $\ell$-ification} of $P(\lambda)$.
\end{definition}

In the following section, we identify an interesting subset of the family of block Kronecker $\ell$-ifications.

\subsection{Block Kronecker companion $\ell$-ifications and Frobenius-like companion forms}\label{sec:companion forms}
In practice, the most important $\ell$-ifications for an $n\times n$ matrix polynomial $P(\lambda)$ as in \eqref{eq:poly} are  the so-called \emph{companion $\ell$-ifications} or \emph{companion forms} \cite{Index Sum Theorem}.
These $\ell$-ifications are introduced in Definition \ref{def:companion form}.
\begin{definition}[Companion form]\label{def:companion form}
Consider an $n\times n$ matrix polynomial $P(\lambda)$ as in \eqref{eq:poly} with degree $d$.
An $\ell$-ification $\mathcal{L}(\lambda)=\sum_{i=0}^\ell \mathcal{L}_i\lambda^i$ of the matrix polynomial $P(\lambda)$ is called a \emph{companion form (or companion $\ell$-ification) of $P(\lambda)$} if, considering the matrix coefficients $\mathcal{L}_i$ as block matrices, each block entry of $\mathcal{L}_i$ is either $0_n$, $I_n$ or $A_i$, for some $i\in\{0,\hdots,d\}$.
\end{definition}

We are aware that Definition \ref{def:companion form} is not the most general definition of companion form considered in the literature (see, for example, \cite[Definition 5.1]{Index Sum Theorem}), but it is the easiest to work with.
It is also worth observing that Definition \ref{def:companion form} implies that companion forms can be constructed from the coefficients of $P(\lambda)$ without performing any arithmetic operation, useful property in numerical computations.

The family of block Kronecker $\ell$-ifications contains many examples of companion forms; see \cite[Section 5.4]{ell-ifications-DPV}. 
This motivates the following definition.
\begin{definition}[Block Kronecker companion form]
Given an $n\times n$ matrix polynomial $P(\lambda)$ as in \eqref{eq:poly} with degree $d$, we will refer to any block Kronecker matrix polynomial
\[
\mathcal{L}(\lambda)=\left[ \begin{array}{c|c}
M(\lambda) & L_\eta(\lambda^\ell)^T\otimes I_n \\ \hline
L_\epsilon(\lambda^\ell)\otimes I_n & \phantom{\Big{(}} 0 \phantom{\Big{)}}
\end{array}\right],
\]
where $M(\lambda)$  satisfies \eqref{eq:P} and each of its block entries is either $0_n$, $I_n$ or $A_i$, for some $i\in\{0,\hdots,d\}$, as a \emph{block Kronecker companion form (or block Kronecker companion $\ell$-ification) of $P(\lambda)$}.
\end{definition}

Let us see some examples of block Kronecker companion forms.
Let $P(\lambda)$ denote an $n\times n$ matrix polynomial as in \eqref{eq:poly} with degree $d$, and let $\ell$ be any divisor of $d$.
Write $k=d/\ell=\epsilon+\eta+1$, for some nonnegative integers $\epsilon,\eta$.
Consider  the matrix polynomials $B_i(\lambda)$ defined in \eqref{eq:B_poly}.
Then, from Theorem \ref{thm:block Kronecker poly}, we can easily check that the block Kronecker matrix polynomial
\[
\mathcal{L}_{\epsilon,\eta}(\lambda) := \left[ \begin{array}{cccc|ccc}
B_k(\lambda) & B_{k-1}(\lambda) & \cdots & B_{\eta+1}(\lambda) & -I_n & 0 & 0  \\
0 & \cdots & 0 & \vdots & \lambda^\ell I_n & \ddots & 0 \\
\vdots & \ddots & \vdots & B_2(\lambda) & 0 & \ddots & -I_n\\
0 & \cdots & 0 & B_1(\lambda) & 0 & 0 & \lambda^\ell I_n \\ \hline
-I_n & \phantom{\Big{(}}\lambda^\ell I_n\phantom{\Big{(}} & 0 & 0 & 0 & \cdots & 0\\
0 & \ddots & \ddots  & 0 & \vdots & \ddots & \vdots \\
0 & 0 & -I_n & \lambda^\ell  I_n & 0 & \cdots & 0
\end{array} \right]
\]
is a block Kronecker companion form of $P(\lambda)$.
Particularizing $\mathcal{L}_{\epsilon,\eta}(\lambda)$ to the pairs $(\epsilon,\eta)=(k-1,0)$ and $(\epsilon,\eta)=(0,k-1)$ yields the so-called \emph{Frobenius-like companion
forms}
\[
C_1^\ell(\lambda):=\begin{bmatrix}
B_k(\lambda) & B_{k-1}(\lambda) & \cdots & B_2(\lambda) & B_1(\lambda) \\
-I_n & \lambda^\ell I_n & 0 & \cdots & 0 \\
0 & \ddots & \ddots & \ddots & \vdots \\
\vdots & \ddots & -I_n & \lambda^\ell I_n & 0 \\
0 & \cdots & 0 & -I_n & \lambda^\ell I_n
\end{bmatrix}
\]
and
\[
C_2^\ell(\lambda):=\begin{bmatrix}
B_k(\lambda) & -I_n & 0 & \cdots & 0 \\
B_{k-1}(\lambda) & \lambda^\ell I_n & \ddots & \ddots & \vdots \\
\vdots & 0 & \ddots & -I_n & 0 \\
B_2(\lambda) & \vdots & \ddots & \lambda^\ell I_n & -I_n \\
B_1(\lambda) & 0 & \cdots & 0 & \lambda^\ell I_n
\end{bmatrix}
\]
introduced and thoroughly analyzed in \cite{Index Sum Theorem}.


\section{One-sided factorizations and eigenvector formulas for block Kronecker $\ell$-ifications}\label{sec:one sided factorization}

One key property of block Kronecker $\ell$-ifications is that they satisfy simple left- and right-sided factorizations.
Obtaining such factorization is the subject of the following section.
These one-sided factorizations will allow us to obtain in Section \ref{sec:eigenvectors} eigenvector formulas for block Kronecker $\ell$-ifications.

\subsection{One-sided factorizations for block Kronecker $\ell$-ifications}

The  goal of this section is to show that every block Kronecker $\ell$-ification $\mathcal{L}(\lambda)$ of a matrix polynomial $P(\lambda)$ satisfies one-sided factorizations of the form
\[
\mathcal{L}(\lambda) H(\lambda) = g\otimes P(\lambda) \quad\mbox{and} \quad 
 G(\lambda)\mathcal{L}(\lambda)  = h^T\otimes P(\lambda),
\]
where $\otimes$ denotes the Kronecker product, for some nonzero vectors $g$ and $h$, and some matrix polynomials $H(\lambda)$ and $G(\lambda)$.
The importance of one-sided factorizations is widely recognized, since they are a useful tool for analyzing nonlinear eigenvalue problems \cite{framework}.

We begin with the auxiliary matrix polynomials that appear repeatedly throughout the following development.
\begin{definition}\label{def:matrices}
For any nonnegative integer $k$ and nonzero natural number $n$, we define the following two block-Toeplitz matrix polynomials:
\begin{equation}\label{eq:Rk}
R_k(\lambda):=
\begin{bmatrix}
I_n & 0 & \cdots & 0 & 0 \\
\lambda I_n & \ddots & \ddots & \vdots & \vdots \\
\vdots & \ddots & I_n & 0 & \vdots \\
\lambda^{k-1}I_n & \cdots & \lambda I_n & I_n & 0 
\end{bmatrix}\in\mathbb{C}[\lambda]^{kn\times (k+1)n},
\end{equation}
and
\begin{equation}\label{eq:Sk}
S_k(\lambda):=
\begin{bmatrix}
0 & \lambda^{k-1}I_n & \lambda^{k-2}I_n & \cdots & I_n \\
\vdots & 0 & \lambda^{k-1}I_n & \ddots & \vdots \\
\vdots & \vdots & \ddots & \ddots & \lambda^{k-2}I_n \\
0 & 0 & \cdots & 0 & \lambda^{k-1}I_n
\end{bmatrix}\in\mathbb{C}[\lambda]^{kn\times (k+1)n},
\end{equation}
with the convention that when $k=0$, both $R_k(\lambda)$ and $S_k(\lambda)$ denote the empty matrix.

For any nonnegative integers  $p$ and $q$, nonzero natural numbers $n$ and $\ell$, and $n(q+1)\times n(p+1)$ matrix polynomial $M(\lambda)$, we define the following two matrix polynomials:
\begin{equation}\label{eq:H}
H(\lambda;p,q,M):=\begin{bmatrix}
\Lambda_p(\lambda^\ell)\otimes I_n \\
R_q(\lambda^\ell)M(\lambda)(\Lambda_p(\lambda^\ell)\otimes I_n)
\end{bmatrix},
\end{equation}
and
\begin{equation}\label{eq:G}
G(\lambda;p,q,M):=\begin{bmatrix}
\lambda^{q\ell} (\Lambda_p(\lambda^\ell)\otimes I_n) \\
-S_q(\lambda^\ell)M(\lambda)(\Lambda_p(\lambda^\ell)\otimes I_n)
\end{bmatrix},
\end{equation}
where $\Lambda_k(\lambda)$, $S_k(\lambda)$ and $R_k(\lambda)$ are defined, respectively in \eqref{eq:Lambda}, \eqref{eq:Sk} and \eqref{eq:Rk}.
Observe that $H(\lambda;p,0,M)=G(\lambda;p,0,M)=\Lambda_{p}(\lambda^\ell)\otimes I_n$.
\end{definition}

Lemma contains some results on the norms of the matrices introduced in Definition \ref{def:matrices} needed for proving the main results of this work.
\begin{lemma}\label{lemma:bound}
Consider the matrix polynomials $\Lambda_k(\lambda)$, $R_k(\lambda)$ and $S_k(\lambda)$ defined in \eqref{eq:Lambda}, \eqref{eq:Rk} and \eqref{eq:Sk}, respectively.
 If $|\lambda|\leq 1$, then
\begin{itemize}
\item[\rm (a1)] $\| \Lambda_k(\lambda^\ell)\otimes I_n \|_2\leq \sqrt{k+1}$, 
 \item[\rm (b1)] $\| R_k(\lambda^\ell)\|_2\leq k$, and
 \item[\rm (c1)] $\| S_k(\lambda^\ell)\|_2\leq k$.
\end{itemize}
If $|\lambda|>1$, then
\begin{itemize}
\item[\rm (a2)] $|\lambda|^{-k\ell}\| \Lambda_k(\lambda^\ell)\otimes I_n \|_2\leq \sqrt{k+1}$, 
 \item[\rm (b2)] $|\lambda|^{-(k-1)\ell}\| R_k(\lambda^\ell)\|_2\leq k$, and
  \item[\rm (c2)] $|\lambda|^{-(k-1)\ell}\| S_k(\lambda^\ell)\|_2\leq k$.
\end{itemize}
\end{lemma}
\begin{proof}
Since $\|\begin{bmatrix} 0 & A \end{bmatrix} \|_2 = \|\begin{bmatrix} A & 0 \end{bmatrix} \|_2 = \|A\|_2$ for any matrix $A$, notice, first, that for computing the 2-norm of the matrices $S_k(\lambda^\ell)$ and $R_k(\lambda^\ell)$ we can ignore their zero columns. 
Recall that the 2-norm is an absolute norm and, thus, a monotone norm.
Then, the six bounds follow from the bound $\|A\|_2 \leq \sqrt{mn}\max_{i,j}\{ |a_{ij}| \}$, which is true for any $m\times n$ matrix $A=(a_{ij})$, the fact that $\|A\otimes B\|_2=\|A\|_2\otimes \|B\|_2$, and the fact that the modulus of the entries of the matrices are all upper bounded by 1.
\end{proof}

In Theorem \ref{thm:right sided fact} we establish two different right-sided factorizations for block Kronecker $\ell$-ifications.
\begin{theorem}\label{thm:right sided fact}{\rm (right-sided factorizations)}
Let $P(\lambda)$ be an $n\times n$ matrix polynomial as in \eqref{eq:poly} of degree $d$.
Assume $d$ is divisible by $\ell$, and let $\mathcal{L}(\lambda)$ as in \eqref{eq:block Kronecker poly} be an $(\epsilon,n,\eta,n)$-block Kronecker $\ell$-ification of $P(\lambda)$.
Then, the following right-sided factorizations hold:
\begin{equation}\label{eq:right-sided1}
\mathcal{L}(\lambda) 
H(\lambda;\epsilon,\eta,M) =\mathcal{L}(\lambda)\begin{bmatrix}
\Lambda_\epsilon(\lambda^\ell)\otimes I_n \\
R_\eta(\lambda^\ell)M(\lambda)(\Lambda_\epsilon(\lambda^\ell)\otimes I_n)
\end{bmatrix} = e_{\eta+1}\otimes P(\lambda),
\end{equation}
and 
\begin{equation}\label{eq:right-sided2}
\mathcal{L}(\lambda) 
G(\lambda;\epsilon,\eta,M)= \mathcal{L}(\lambda)\begin{bmatrix}
\lambda^{\eta\ell} (\Lambda_\epsilon(\lambda^\ell)\otimes I_n) \\
-S_\eta(\lambda^\ell)M(\lambda)(\Lambda_\epsilon(\lambda^\ell)\otimes I_n)
\end{bmatrix} = e_{1}\otimes P(\lambda),
\end{equation}
where $e_j$ denotes the $j$th column of the $d/\ell\times d/\ell$ identity matrix.
\end{theorem}

\begin{proof}
Recall that the matrix polynomial $M(\lambda)$ in the (1,1) block of $\mathcal{L}(\lambda)$ satisfies  $ (\Lambda_\eta(\lambda^\ell)^T\otimes I_n)M(\lambda)(\Lambda_\epsilon(\lambda^\ell)\otimes I_n)=P(\lambda)$, as this will be important.
The proofs of \eqref{eq:right-sided1} and \eqref{eq:right-sided2} consist of  direct verifications.

We begin by proving \eqref{eq:right-sided1}. 
First, notice $(L_\epsilon(\lambda^\ell)\otimes I_n)(\Lambda_\epsilon(\lambda^\ell)\otimes I_n)=0$.
This implies that the bottom $n\epsilon$ rows of $\mathcal{L}(\lambda) 
H(\lambda;\epsilon,\eta,M)$ are all equal to zero. 
The remaining rows (upper $(\eta+1)n$ rows) are equal to 
\begin{align*}
&M(\lambda)(\Lambda_\epsilon(\lambda^\ell)\otimes I_n) + (L_\eta(\lambda^\ell)^T\otimes I_n)R_\eta(\lambda^\ell)M(\lambda)(\Lambda_\epsilon(\lambda^\ell)\otimes I_n) = \\
&\left(I_{(\eta+1)n}+(L_\eta(\lambda^\ell)^T\otimes I_n)R_\eta(\lambda^\ell)\right)M(\lambda)(\Lambda_\epsilon(\lambda^\ell)\otimes I_n) = \\
&\left(I_{(\eta+1)n}-I_{(\eta+1)n}+\begin{bmatrix} 0 \\ \Lambda_\eta(\lambda^\ell)^T\otimes I_n \end{bmatrix} \right)M(\lambda)(\Lambda_\epsilon(\lambda^\ell)\otimes I_n) = \\
&\begin{bmatrix}
0 \\ (\Lambda_\eta(\lambda^\ell)^T\otimes I_n)M(\lambda)(\Lambda_\epsilon(\lambda^\ell)\otimes I_n)
\end{bmatrix} = \begin{bmatrix}
0 \\ P(\lambda)
\end{bmatrix},
\end{align*}
from where the result readily follows.

Next, we prove \eqref{eq:right-sided2}. 
From  $(L_\epsilon(\lambda^\ell)\otimes I_n)(\Lambda_\epsilon(\lambda^\ell)\otimes I_n)=0$, we obtain that the bottom $n\epsilon$ rows of $\mathcal{L}(\lambda) 
G(\lambda;\epsilon,\eta,M)$ are all equal to zero. 
The remaining rows (upper $(\eta+1)n$ rows) are equal to 
\begin{align*}
&\lambda^{\eta\ell} M(\lambda)(\Lambda_\epsilon(\lambda^\ell)\otimes I_n) - (L_\eta(\lambda^\ell)^T\otimes I_n)S_\eta(\lambda^\ell)M(\lambda)(\Lambda_\epsilon(\lambda^\ell)\otimes I_n) = \\
&\left(\lambda^{\eta\ell} I_{(\eta+1)n}-(L_\eta(\lambda^\ell)^T\otimes I_n)S_\eta(\lambda^\ell)\right)M(\lambda)(\Lambda_\epsilon(\lambda^\ell)\otimes I_n) = \\
&\left(\lambda^{\eta\ell} I_{(\eta+1)n}-\lambda^{\eta\ell} I_{(\eta+1)n}+\begin{bmatrix} \Lambda_\eta(\lambda^\ell)^T\otimes I_n \\ 0 \end{bmatrix} \right)M(\lambda)(\Lambda_\epsilon(\lambda^\ell)\otimes I_n) = \\
&\begin{bmatrix}
 (\Lambda_\eta(\lambda^\ell)^T\otimes I_n)M(\lambda)(\Lambda_\epsilon(\lambda^\ell)\otimes I_n) \\ 0
\end{bmatrix} = \begin{bmatrix}
P(\lambda) \\ 0
\end{bmatrix},
\end{align*}
which implies the desired result.
\end{proof}

We obtain in Theorem  \ref{thm:left sided fact}  left-sided factorizations for block Kronecker $\ell$-ifications, analogues of the right-sided factorizations in Theorem \ref{thm:right sided fact}.
\begin{theorem}\label{thm:left sided fact}{\rm (left-sided factorizations)}
Let $P(\lambda)$ be an $n\times n$ matrix polynomial as in \eqref{eq:poly} of degree $d$.
Assume $d$ is divisible by $\ell$, and let $\mathcal{L}(\lambda)$ as in \eqref{eq:block Kronecker poly} be an $(\epsilon,n,\eta,n)$-block Kronecker $\ell$-ification of $P(\lambda)$.
Then, the following left-sided factorizations hold:
\begin{equation}\label{eq:left-sided1}
H(\lambda;\eta,\epsilon,M^T)^T
\mathcal{L}(\lambda)  = \begin{bmatrix}
\Lambda_\eta(\lambda^\ell)\otimes I_n \\
R_\epsilon(\lambda^\ell)M(\lambda)^T(\Lambda_\eta(\lambda^\ell)\otimes I_n)
\end{bmatrix}^T\mathcal{L}(\lambda)
 = e_{\epsilon+1}^T\otimes P(\lambda),
\end{equation}
and 
\begin{equation}\label{eq:left-sided2}
G(\lambda;\eta,\epsilon,M^T)^T
\mathcal{L}(\lambda) =
\begin{bmatrix}
\lambda^{\epsilon\ell} (\Lambda_\eta(\lambda^\ell)\otimes I_n) \\
-S_\epsilon(\lambda^\ell)M(\lambda)^T(\Lambda_\eta(\lambda^\ell)\otimes I_n)
\end{bmatrix}^T\mathcal{L}(\lambda)=
  e_{1}^T\otimes P(\lambda),
\end{equation}
where $e_j$ denotes the $j$th column of the $d/\ell\times d/\ell$ identity matrix.
\end{theorem}
\begin{proof}
Notice that $\mathcal{L}(\lambda)^T$ is a block Kronecker matrix polynomial with the roles of $\epsilon$ and $\eta$ interchanged, and with $(1,1)$ block equal to $M(\lambda)^T$.
Hence, $\mathcal{L}(\lambda)^T$ is an $(\eta,n,\epsilon,n)$-block Kronecker $\ell$-ification of the matrix polynomial
\[
(\Lambda_\epsilon(\lambda^\ell)^T\otimes I_n)M(\lambda)^T(\Lambda_\eta(\lambda^\ell)\otimes I_n)=P(\lambda)^T.
\]
Then, the left-sided factorizations \eqref{eq:left-sided1} and \eqref{eq:left-sided2} can be obtained by applying Theorem \ref{thm:right sided fact} to  $\mathcal{L}(\lambda)^T$.
\end{proof}

The one-sided factorizations in Theorems \ref{thm:right sided fact} and \ref{thm:left sided fact}  allow us in the following section to obtain simple formulas for the eigenvectors of block Kronecker $\ell$-ifications in terms of the eigenvectors of the matrix polynomial $P(\lambda)$.
These formulas are a key feature of block Kronecker $\ell$-ifications, and they will play a key role in all of our results.
\subsection{Eigenvector formulas for block Kronecker $\ell$-ifications}\label{sec:eigenvectors}

Theorem \ref{thm:right eigenvectors} establishes two relations between right eigenvectors of $P(\lambda)$ and right eigenvectors of a block Kronecker $\ell$-ification of $P(\lambda)$.

\begin{theorem}\label{thm:right eigenvectors}{\rm (right eigenvector formulas)}
Let $P(\lambda)$ be an $n\times n$ matrix polynomial as in \eqref{eq:poly} of degree $d$.
Assume $d$ is divisible by $\ell$, and let $\mathcal{L}(\lambda)$ as in \eqref{eq:block Kronecker poly} be an $(\epsilon,n,\eta,n)$-block Kronecker $\ell$-ification of $P(\lambda)$.
The following statements hold.
\begin{itemize}
\item[\rm (a)] Let $\lambda_0$ be a finite eigenvalue of $P(\lambda)$. 
A vector $z$ is a right eigenvector of $\mathcal{L}(\lambda)$ associated with $\lambda_0$ if and only if 
\begin{equation}\label{eq:right eigenvector1}
z = H(\lambda_0;\epsilon,\eta,M)x=\begin{bmatrix}
\Lambda_\epsilon(\lambda_0^\ell)\otimes I_n \\
R_\eta(\lambda_0^\ell)M(\lambda_0)(\Lambda_\epsilon(\lambda_0^\ell)\otimes I_n)
\end{bmatrix}x,
\end{equation}
for some right eigenvector $x$ of $P(\lambda)$ associated with $\lambda_0$.
\item[\rm (b)] Let $\lambda_0$ be a finite \emph{nonzero} eigenvalue of $P(\lambda)$. 
A vector $z$ is a right eigenvector of $\mathcal{L}(\lambda)$ associated with $\lambda_0$ if and only if 
\begin{equation}\label{eq:right eigenvector2}
z =G(\lambda_0;\epsilon,\eta,M)x= \begin{bmatrix}
\lambda_0^{\eta\ell} (\Lambda_\epsilon(\lambda_0^\ell)\otimes I_n) \\
-S_\eta(\lambda_0^\ell)M(\lambda_0)(\Lambda_\epsilon(\lambda_0^\ell)\otimes I_n)
\end{bmatrix}x,
\end{equation}
for some right eigenvector $x$ of $P(\lambda)$ associated with $\lambda_0$.
\end{itemize}
\end{theorem}
\begin{proof}
The eigenvector formulas are an immediate consequence of the right-sided factorizations \eqref{eq:right-sided1} and \eqref{eq:right-sided2}.
We only prove part (a), since part (b) follows from a similar argument using \eqref{eq:right-sided2} instead of \eqref{eq:right-sided1}. 

Let $x$ be a right eigenvector of $P(\lambda)$ with eigenvalue $\lambda_0$, and let $z$ be the vector \eqref{eq:right eigenvector1}. 
Notice that if $x$ is a nonzero vector, so is $z$.
Evaluating \eqref{eq:right-sided1} at the eigenvalue $\lambda_0$, multiplying from the right by $x$, and using $P(\lambda_0)x=0$, yields 
\[
\mathcal{L}(\lambda_0)
\begin{bmatrix}
\Lambda_\epsilon(\lambda_0^\ell)\otimes I_n \\
R_\eta(\lambda_0^\ell)M(\lambda_0)(\Lambda_\epsilon(\lambda_0^\ell)\otimes I_n)
\end{bmatrix}x = \mathcal{L}(\lambda_0)z=e_{\eta+1}\otimes (P(\lambda_0)x) = 0.
\]
Hence, $z$ is a right eigenvector of $\mathcal{L}(\lambda)$ with  eigenvalue $\lambda_0$.

Let $z$ be a right eigenvector of $\mathcal{L}(\lambda)$ with eigenvalue $\lambda_0$. 
Assume $\lambda_0$ as an eigenvalue of $\mathcal{L}(\lambda)$ has geometric multiplicity $m$.
Since $\mathcal{L}(\lambda)$ is a strong $\ell$-ification of $P(\lambda)$, the geometric multiplicity of $\lambda_0$ as an eigenvalue of $P(\lambda)$ is also $m$.
Let $x_1,\hdots,x_m$ be linearly independent eigenvectors of $P(\lambda)$ with eigenvalue $\lambda_0$, and define $z_i = H(\lambda_0;\epsilon,\eta,M)x_i$, for $i=1,\hdots,m$.
Since $H(\lambda_0)$ has full column rank, the vectors $z_1,\hdots,z_m$ are linearly independent eigenvectors for $\mathcal{L}(\lambda)$.
Hence, the vector $z$ must be a linear combination of $z_1,\hdots,z_m$, that is,
\[
z = \sum_{i=1}^m c_i z_i = \sum_{i=1}^m c_i H(\lambda_0;\epsilon,\eta,M)x_i = H(\lambda_0;\epsilon,\eta,M) \sum_{i=1}^m c_ix_i.
\]
Therefore, $z$ is of the form $H(\lambda_0;\epsilon,\eta,M)x$ for some eigenvector $x$ for $P(\lambda)$ with eigenvalue $\lambda_0$.
\end{proof}

Theorem \ref{thm:left eigenvectors} establishes two relations between left eigenvectors of $P(\lambda)$ and left eigenvectors of a block Kronecker $\ell$-ification of $P(\lambda)$.
Here, the scalar $\overline{\lambda}$ denotes the complex conjugate of the complex number $\lambda$.

\begin{theorem}\label{thm:left eigenvectors}{\rm (left eigenvector formulas)}
Let $P(\lambda)$ be an $n\times n$ matrix polynomial as in \eqref{eq:poly}. 
Assume its degree $d$ is divisible by $\ell$, and let $\mathcal{L}(\lambda)$ as in \eqref{eq:block Kronecker poly} be an $(\epsilon,n,\eta,n)$-block Kronecker $\ell$-ification of $P(\lambda)$.
The following statements hold.
\begin{itemize}
\item[\rm (a)] Let $\lambda_0$ be a simple eigenvalue of $P(\lambda)$. 
A vector $w$ is a left eigenvector of $\mathcal{L}(\lambda)$ associated with $\lambda_0$ if and only if 
\[
w = H(\overline{\lambda_0};\eta,\epsilon,M^*)y = \begin{bmatrix}
\Lambda_\eta(\overline{\lambda_0^\ell})\otimes I_n \\
R_\eta(\overline{\lambda_0^\ell})M(\lambda_0)^*(\Lambda_\eta(\overline{\lambda_0^\ell})\otimes I_n)
\end{bmatrix}y,
\]
for some left eigenvector $y$ of $P(\lambda)$ associated with $\lambda_0$.
\item[\rm (b)] Let $\lambda_0$ be a simple \emph{nonzero} eigenvalue of $P(\lambda)$. 
A vector $w$ is a left eigenvector of $\mathcal{L}(\lambda)$ associated with $\lambda_0$ if and only if 
\[
w = G(\overline{\lambda_0};\eta,\epsilon,M^*)y =  \begin{bmatrix}
\overline{\lambda_0^{\epsilon\ell}}(\Lambda_\eta(\overline{\lambda_0^\ell})\otimes I_n)\\ -S_\epsilon(\overline{\lambda_0^\ell})M(\lambda)^*(\Lambda_\eta(\overline{\lambda_0^\ell})\otimes I_n)
\end{bmatrix}y,
\]
for some left eigenvector $y$ of $P(\lambda)$ associated with $\lambda_0$.
\end{itemize}
\end{theorem}
\begin{proof}
Parts (a) and (b) follow from the left-sided factorizations in Theorem \ref{thm:left sided fact}.
The proof is nearly identical to the proof of Theorem \ref{thm:right eigenvectors}, so it is omitted.
\end{proof}

\section{The conditioning of block Kronecker $\ell$-ifications}
\label{sec:conditioning1}

Let $\lambda_0$ be a simple, finite, nonzero eigenvalue of an $n\times n$ matrix polynomial $P(\lambda)$ as in \eqref{eq:poly}, and let $\mathcal{L}(\lambda)=\sum_{i=0}^\ell \mathcal{L}_i\lambda^i$ be a block Kronecker $\ell$-ification of $P(\lambda)$.
Note that $\lambda_0$ as an eigenvalue of $\mathcal{L}(\lambda)$ is also simple, because strong $\ell$-ifications preserve geometric multiplicities.
Let $x$ and $y$ denote right and left eigenvectors of $P(\lambda)$, and let $z$ and $w$ denote right and left eigenvectors of $\mathcal{L}(\lambda)$, all corresponding to the eigenvalue $\lambda_0$. 
By \eqref{eq:condition number}, we have eigenvalue condition numbers given by 
\begin{align*}
&\mathrm{norm\,cond}_{P}(\lambda_0)=\frac{\left(\|\begin{bmatrix} A_0 & \cdots  A_d \end{bmatrix}\|_2 \sum_{i=0}^d|\lambda_0|^i \right)\|x\|_2\|y\|_2}{|\lambda_0|\cdot |y^*P^\prime(\lambda_0)x|},\\
&\mathrm{coeff\,cond}_{P}(\lambda_0)=\frac{\left( \sum_{i=0}^d|\lambda_0|^i\|A_i\|_2 \right)\|x\|_2\|y\|_2}{|\lambda_0|\cdot |y^*P^\prime(\lambda_0)x|},  \quad \quad \mbox{and} \\
&\mathrm{coeff\,cond}_{\mathcal{L}}(\lambda_0)=\frac{\left(\sum_{i=0}^\ell|\lambda_0|^i\|\mathcal{L}_i\|_2 \right)\|z\|_2\|w\|_2}{|\lambda_0|\cdot |w^*\mathcal{L}^\prime(\lambda_0)z|}. 
\end{align*}

It is known that an $\ell$-ification $\mathcal{L}(\lambda)$ may alter the conditioning of the eigenvalue problem quite drastically (see the numerical experiments in Section \ref{sec:experiments}).
For this reason, our aim in this section is to study the ratios
\begin{align}
\label{eq:ratios of cond numbers2}
&\frac{\mathrm{coeff\,cond}_\mathcal{L}(\lambda_0)}{\mathrm{coeff\,cond}_P(\lambda_0)} = \frac{\left( \sum_{i=0}^\ell |\lambda_0|^i\|\mathcal{L}_i\|_2 \right)}{\left( \sum_{i=0}^d |\lambda_0|^i\|A_i\|_2\ \right)}\frac{|y^*P^\prime(\lambda_0)x|}{|w^*\mathcal{L}^\prime(\lambda_0)z|}\frac{\|z\|_2\|w\|_2}{\|x\|_2\|y\|_2} \quad \mbox{and} \\
\label{eq:ratios of cond numbers}
&\frac{\mathrm{coeff\,cond}_\mathcal{L}(\lambda_0)}{\mathrm{norm\,cond}_P(\lambda_0)} = \frac{\left( \sum_{i=0}^\ell |\lambda_0|^i\|\mathcal{L}_i\|_2 \right)}{\left(\|\begin{bmatrix} A_0 & \cdots & A_d \end{bmatrix}\|_2 \sum_{i=0}^d |\lambda_0|^i\ \right)}\frac{|y^*P^\prime(\lambda_0)x|}{|w^*\mathcal{L}^\prime(\lambda_0)z|}\frac{\|z\|_2\|w\|_2}{\|x\|_2\|y\|_2}.
\end{align}
More specifically, our immediate goals are, first, to obtain upper bounds for \eqref{eq:ratios of cond numbers2} and \eqref{eq:ratios of cond numbers}, and, second, to study under which conditions these upper bounds are moderate for all the eigenvalues of $P(\lambda)$.

We start with Lemma \ref{lemma:relating cond numb}, which implies a close relation between the conditioning of $P(\lambda)$ and the conditioning of its block Kronecker $\ell$-ifications.
\begin{lemma}\label{lemma:relating cond numb}
Let $P(\lambda)$ be an $n\times n$ matrix polynomial as in \eqref{eq:poly} of degree $d$.
Assume $d$ is divisible by $\ell$, and let $\mathcal{L}(\lambda)$  as in \eqref{eq:block Kronecker poly}  be an $(\epsilon,n,\eta,n)$-block Kronecker $\ell$-ification of $P(\lambda)$.
If $\lambda_0$ is a simple and finite eigenvalue of $P(\lambda)$, with right and left eigenvectors $x$ and $y$, respectively, then the following statements hold.
\begin{itemize}
\item[\rm(a)] The vectors $z=H(\lambda_0;\epsilon,\eta,M)x$ and $H(\overline{\lambda_0};\eta,\epsilon,M^*)y$ are, respectively, right and left eigenvectors of $\mathcal{L}(\lambda)$ with eigenvalue $\lambda_0$, and 
\begin{equation}\label{eq:relating cond numb 1}
|w^*\mathcal{L}^\prime(\lambda_0)z| = |y^* P^\prime(\lambda_0) x|,
\end{equation}
where $H(\lambda;p,q,M)$ is defined in \eqref{eq:H}.
\item[\rm(b)] Assume $\lambda_0$ is, in addition, nonzero. 
The vectors $z=G(\lambda_0;\epsilon,\eta,M)x$ and $G(\overline{\lambda_0};\eta,\epsilon,M^*)y$ are, respectively, right and left eigenvectors of $\mathcal{L}(\lambda)$ with eigenvalue $\lambda_0$, and 
\begin{equation}\label{eq:relating cond numb 2}
|w^*\mathcal{L}^\prime(\lambda_0)z| = |\lambda_0|^{d-\ell} |y^* P^\prime(\lambda_0) x|,
\end{equation}
where $G(\lambda;p,q,M)$ is defined in \eqref{eq:G}.
\end{itemize}
\end{lemma}
\begin{proof}
We first prove \eqref{eq:relating cond numb 1}.
Observe that part (a) of Theorems \ref{thm:right eigenvectors} and \ref{thm:left eigenvectors} say that $z=H(\lambda_0;\epsilon,\eta,M)x$ and $w=H(\overline{\lambda_0};\eta,\epsilon,M^*)y$ are, respectively, right and left eigenvectors of $ \mathcal{L}(\lambda_0)$ associated with $\lambda_0$.
Now, differentiating \eqref{eq:right-sided1} with respect to $\lambda$ gives
\[
\mathcal{L}^\prime (\lambda)H(\lambda;\epsilon,\eta,M) + \mathcal{L}(\lambda)H^\prime(\lambda;\epsilon,\eta,M)=e_{\eta+1}\otimes P^\prime(\lambda).
\]
Evaluating the equation above at the eigenvalue $\lambda_0$, multiplying from the left by $w^*$, multiplying from the right by $x$, and using $w^*\mathcal{L}(\lambda_0)=0$ produces
\[
w^* \mathcal{L}^\prime(\lambda_0)z = w^*(e_{\eta+1}\otimes (P^\prime(\lambda_0)x)) = y^*P^\prime(\lambda_0)x,
\]
where we have used that the $\eta+1$ block-entry of $H(\overline{\lambda_0};\eta,\epsilon,M^*)y$ equals $y$. 

Next, we prove \eqref{eq:relating cond numb 2}.
From part (b) of Theorems \ref{thm:right eigenvectors} and \ref{thm:left eigenvectors}, it follows that $z:=G(\lambda_0;\epsilon,\eta,M)x$ and $w:=G(\overline{\lambda_0};\eta,\epsilon,M^*)y$ are, respectively, right and left eigenvectors of $ \mathcal{L}(\lambda)$ associated with $\lambda_0$.
Differentiating \eqref{eq:right-sided2} with respect to $\lambda$ gives
\[
\mathcal{L}^\prime (\lambda)G(\lambda;\epsilon,\eta,M) + \mathcal{L}(\lambda)G^\prime(\lambda;\epsilon,\eta,M)=e_{1}\otimes P^\prime(\lambda).
\]
Evaluating the equation above at the eigenvalue $\lambda_0$, multiplying from the left by $w^*$, multiplying from the right by $x$, and using $w^*\mathcal{L}(\lambda_0)=0$ yields
\[
w^* \mathcal{L}^\prime(\lambda_0)z = w^*(e_{1}\otimes (P^\prime(\lambda_0)x)) = \overline{\lambda_0}^{d-\ell} y^*P^\prime(\lambda_0)x,
\]
where we have used that the first block-entry of $G(\overline{\lambda_0};\eta,\epsilon,M^*)y$ equals $\overline{\lambda_0}^{(\epsilon+\eta)\ell}y = \overline{\lambda_0}^{d-\ell}y$.
\end{proof}


The expressions \eqref{eq:relating cond numb 1} and \eqref{eq:relating cond numb 2} can now be used to investigate the size of the ratios \eqref{eq:ratios of cond numbers2} and \eqref{eq:ratios of cond numbers}.

\begin{theorem}\label{thm:ratio L-P}
Let $P(\lambda)$ be an $n\times n$ matrix polynomial as in \eqref{eq:poly} of degree $d$.
Assume $d$ is divisible by $\ell$, and let $ \mathcal{L}(\lambda)=\sum_{i=0}^\ell \mathcal{L}_i\lambda^i$ as in \eqref{eq:block Kronecker poly} be an $(\epsilon,n,\eta,n)$-block Kronecker $\ell$-ification of $P(\lambda)$.
If $\lambda_0$ is a simple, finite, nonzero eigenvalue of $P(\lambda)$, then
\begin{align}\label{eq:ratio L-P}
\begin{split}
\dfrac{\mathrm{coeff\,cond}_{\mathcal{L}}(\lambda_0)}{\mathrm{coeff\,cond}_{P}(\lambda_0)}\leq &
\dfrac{2\max\{1,\max_{i=0:\ell}\{\|M_i\|_2\}\}}{\min\{ \|A_0\|_2,\|A_d\|_2 \}}
(\ell+1)(\epsilon+1)^{1/2}(\eta+1)^{1/2}\times\\ 
&\left(1+\epsilon^2(\ell+1)\sum_{i=0}^\ell\|M_i\|_2^2\right)^{1/2}
\left(1+\eta^2(\ell+1)\sum_{i=0}^\ell\|M_i\|_2^2\right)^{1/2},
\end{split}
\end{align}
and
\begin{align}\label{eq:ratio L-P norm}
\begin{split}
\dfrac{\mathrm{coeff\,cond}_{\mathcal{L}}(\lambda_0)}{\mathrm{norm\, cond}_{P}(\lambda_0)}\leq &
\dfrac{2\max\{1,\max_{i=0:\ell}\{\|M_i\|_2\}\}}{\| \begin{bmatrix} A_0 & \cdots & A_d \end{bmatrix} \|_2}
(\ell+1)(\epsilon+1)^{1/2}(\eta+1)^{1/2}\times\\ 
&\left(1+\epsilon^2(\ell+1)\sum_{i=0}^\ell\|M_i\|_2^2\right)^{1/2}
\left(1+\eta^2(\ell+1)\sum_{i=0}^\ell\|M_i\|_2^2\right)^{1/2}.
\end{split}
\end{align}
\end{theorem}
\begin{proof}
We only prove \eqref{eq:ratio L-P}, since the proof of \eqref{eq:ratio L-P norm} is entirely analogous.

Let $x$ and $y$ be, respectively, right and left eigenvectors of $P(\lambda)$ associated with the eigenvalue $\lambda_0$.
We need to distinguish two cases, namely, the case when $|\lambda_0|\leq 1$ and the case when $|\lambda_0|>1$.

Let us assume, first,  $|\lambda_0|\leq 1$
From part (a) in Lemma \ref{lemma:relating cond numb}, we obtain that the ratio \eqref{eq:ratios of cond numbers2} equals
\begin{equation}
\label{eq:ratio1}\frac{\mathrm{coeff\,cond}_{L}(\lambda_0)}{\mathrm{coeff\,cond}_{P}(\lambda_0)}=
\frac{\sum_{i=0}^\ell |\lambda_0|^i \| \mathcal{L}_i \|_2}{\sum_{i=0}^d |\lambda_0|^i \| A_i \|_2}\frac{\|H(\lambda_0;\epsilon,\eta,M)x\|_2\|H(\overline{\lambda_0};\eta,\epsilon,M^*)y\|_2}{\|x\|_2\|y\|_2},
\end{equation}
where $H(\lambda;p,q,M)$ is defined in \eqref{eq:H}.
Now, since $|\lambda_0|\leq 1$, we have
\begin{align*}
\frac{\sum_{i=0}^\ell |\lambda_0|^i \| \mathcal{L}_i \|_2}{\sum_{i=0}^d |\lambda_0|^i \| A_i \|_2}\leq  \frac{\sum_{i=0}^\ell \| \mathcal{L}_i \|_2}{\|A_0\|_2} \leq & \frac{ (\ell+1)\max_{i=0:\ell}\{ \| \mathcal{L}_i \|_2\}}{\min\{\|A_0\|_2,\|A_d\|_2\}} \leq \\
 & \frac{ 2(\ell+1)\max\{1,\max_{i=0:\ell}\{ \| \mathcal{L}_i \|_2\}\}}{\min\{\|A_0\|_2,\|A_d\|_2\}},
\end{align*}
where  the last inequality follows from Lemma \ref{lemma:norm bound}.
Then, notice 
\begin{align}\label{eq:Y1}
\begin{split}
\frac{\|H(\lambda_0;\epsilon,\eta,M)x\|_2^2}{\|x\|_2^2}=&
\frac{1}{\|x\|_2^2}\,{\Big\|} \begin{bmatrix} \Lambda_\epsilon(\lambda_0^\ell)\otimes x \\ R_\eta(\lambda_0^\ell)M(\lambda_0)(\Lambda_\epsilon(\lambda_0^\ell)\otimes x) \end{bmatrix} {\Big\|}_2^2 \leq \\
&\frac{1}{\|x\|_2^2} \left( \|\Lambda_\epsilon(\lambda_0^\ell)\|_2^2\|x\|_2^2+\|R_\eta(\lambda_0^\ell)\|_2^2 \| M(\lambda_0)\|_2^2 \|\Lambda_\epsilon(\lambda_0^\ell)\|_2^2\|x\|_2^2 \right)\leq \\
&(\epsilon+1)+\eta^2(\epsilon+1)\left(\sum_{i=0}^\ell|\lambda_0|^{i} \|M_i\|_2\right)^2 \leq \\
&(\epsilon+1)\left(1+\eta^2(\ell+1)\sum_{i=0}^\ell \|M_i\|_2^2\right),
\end{split}
\end{align}
where we have used the inequality $(\sum_{i=0}^k |a_i|)^2 \leq (k+1)\sum_{i=0}^k |a_i|^2$ for obtaining the last inequality above.
An identical argument gives the upper bound
\begin{equation}\label{eq:Y2}
\frac{\|H(\overline{\lambda_0};\eta,\epsilon,M^*)y\|_2^2}{\|y\|_2^2} \leq
(\eta+1)\left(1+\epsilon^2(\ell+1)\sum_{i=0}^\ell \|M_i\|_2^2\right).
\end{equation}
Combining the previous three inequalities, yields the desired result.

Next, let us assume $|\lambda_0|> 1$.
By part (b) in Lemma \ref{lemma:relating cond numb}, the ratio \eqref{eq:ratios of cond numbers2} can be rewritten as
\begin{equation}
\label{eq:ratio2}\frac{\mathrm{coeff\,cond}_{\mathcal{L}}(\lambda_0)}{\mathrm{coeff\,cond}_{P}(\lambda_0)}=
\frac{\sum_{i=0}^\ell |\lambda_0|^i \| \mathcal{L}_i \|_2}{\sum_{i=0}^d |\lambda_0|^i \| A_i \|_2}\frac{\|G(\lambda_0;\epsilon,\eta,M)x\|_2\|G(\overline{\lambda_0};\eta,\epsilon,M^*)y\|_2}{|\lambda_0|^{d-\ell}\|x\|_2\|y\|_2}.
\end{equation}
Since $|\lambda_0|> 1$, we easily see that
\begin{align*}
\frac{\sum_{i=0}^\ell |\lambda_0|^i \| \mathcal{L}_i \|_2}{\sum_{i=0}^d |\lambda_0|^i \| A_i \|_2}\leq \frac{\sum_{i=0}^\ell |\lambda_0|^i\| \mathcal{L}_i \|_2}{|\lambda_0|^d \|A_d\|_2}\leq &\frac{(\ell+1)\max_{i=0:d}\{\|\mathcal{L}_i\|_2\}}{|\lambda_0|^{d-\ell}\min\{\|A_0\|_2,\|A_d\|_2\}} \leq \\
 & \frac{2(\ell+1)\max\{1,\max_{i=0:d}\{\|M_i\|_2\}\}}{|\lambda_0|^{d-\ell}\min\{\|A_0\|_2,\|A_d\|_2\}} .
\end{align*}
Hence, we have the upper bound
\begin{align*}
&\frac{\mathrm{coeff\,cond}_{\mathcal{L}}(\lambda_0)}{\mathrm{coeff\,cond}_{P}(\lambda_0)}\leq \\
&\hspace{1.6cm}
\frac{2(\ell+1)\max\{1,\max_{i=0:d}\{\|M_i\|_2\}\}}{\min\{\|A_0\|_2,\|A_d\|_2\}}\frac{\|G(\lambda_0;\epsilon,\eta,M)x\|_2}{|\lambda_0|^{d-\ell}\|x\|_2}\frac{\|G(\overline{\lambda_0};\eta,\epsilon,M^*)y\|_2}{|\lambda_0|^{d-\ell}\|y\|_2}.
\end{align*}
Then, recall $d-\ell=\epsilon\ell+\eta\ell$ and the bounds in Lemma \ref{lemma:bound}, and notice
\begin{align}\label{eq:X1}
\begin{split}
\frac{\|G(\lambda_0;\epsilon,\eta,M)x\|_2^2}{|\lambda_0|^{2(d-\ell)}\|x\|_2^2}=&
\frac{1}{\|x\|_2^2}\,{\Big\|} \begin{bmatrix} \lambda_0^{-\epsilon\ell}\Lambda_\epsilon(\lambda_0^\ell)\otimes x \\ -\lambda_0^{\ell-d}S_\eta(\lambda_0^\ell)M(\lambda_0)(\Lambda_\epsilon(\lambda_0^\ell)\otimes x) \end{bmatrix} {\Big\|}_2^2 \leq \\
&\frac{1}{\|x\|_2^2} \left( \|\lambda_0^{-\epsilon\ell}\Lambda_\epsilon(\lambda_0^\ell)\|_2^2\|x\|_2^2+\right. \\
&\hspace{1.5cm}\left.\|\lambda_0^{-\eta\ell+\ell}S_\eta(\lambda_0^\ell)\|_2^2 \|\lambda_0^{-\ell} M(\lambda_0)\|_2^2 \|\lambda_0^{-\epsilon\ell}\Lambda_\epsilon(\lambda_0^\ell)\|_2^2\|x\|_2^2 \right)\leq \\
&(\epsilon+1)+\eta^2(\epsilon+1)\left(\sum_{i=0}^\ell \|M_i\|_2\right)^2 \leq \\
&(\epsilon+1)\left(1+\eta^2(\ell+1)\sum_{i=0}^\ell \|M_i\|_2^2\right).
\end{split}
\end{align}
Analogously, we can obtain the upper bound
\begin{equation}\label{eq:X2}
\frac{\|G(\overline{\lambda_0};\eta,\epsilon,M^*)y\|_2^2}{|\lambda_0|^{2(d-\ell)}\|y\|_2^2}\leq 
(\eta+1)\left(1+\epsilon^2(\ell+1)\sum_{i=0}^\ell \|M_i\|_2^2\right).
\end{equation}
The desired result now follows easily.
\end{proof}


Were the bounds in Theorem \ref{thm:ratio L-P}  moderate, the block Kronecker $\ell$-ification $\mathcal{L}(\lambda)$ would be about as well conditioned as the polynomial $P(\lambda)$ itself.
In Proposition \ref{prop:bound norm M}, we show that a necessary condition for having moderate bounds \eqref{eq:ratios of cond numbers2} and \eqref{eq:ratios of cond numbers} is  $\max_{i=0:d}\{\|A_i\|_2\}$ to be moderate. 

%
\begin{proposition}\label{prop:bound norm M}
Let $P(\lambda)$ be an $n\times n$ matrix polynomial as in \eqref{eq:poly} of degree $d=k(\epsilon+\eta+1)$, and let $M(\lambda)=\sum_{i=0}^\ell M_i\lambda^i$ be an $(\eta+1)n\times (\epsilon+1)n$ grade-$\ell$ matrix polynomial satisfying \eqref{eq:P}.
Then, 
\[
\max_{i=0:\ell}\{ \|M_i\|_2 \} \geq \frac{1}{2\max\{\epsilon+1,\eta+1\}}\max_{i=0:d}\{\|A_i\|_2\}.
\]
Hence, the upper bounds \eqref{eq:ratios of cond numbers2} and \eqref{eq:ratios of cond numbers} can potentially show a cubic dependence on  $\max_{i=0:d}\{\|A_i\|_2\}$.
\end{proposition}
\begin{proof}
By part (iii) in Theorem \ref{thm:characterization of solutions}, notice that each matrix coefficient $A_i$ satisfies either an equation of the form $A_i = \sum_{i+j=c}[M_\ell]_{ij} + \sum_{i+j=c-1}[M_0]_{ij}$, for some constant $c$, or of the form $A_i = \sum_{i+j=c_1}[M_{c_2}]_{ij}$, for some constants $c_1$ and $c_2$. 
In the former case, we have
\begin{align*}
\|A_i\|_2 \leq & \sum_{i+j=c}\|[M_\ell]_{ij}\|_2 + \sum_{i+j=c-1}\|[M_0]_{ij} \|_2 \leq \\
& \max\{\epsilon+1,\eta+1\}\left(\max_{ij}\{\|[M_\ell]_{ij}\|_2 \}+\max_{ij}\{\|[M_0]_{ij}\|_2 \} \right) \leq \\
&\max\{\epsilon+1,\eta+1\}\left(\|M_\ell\|_2+\|M_0\|_2 \right)\leq 2\max\{\epsilon+1,\eta+1\}\max_{i=0:\ell}\{|M_i\|_2\}.
\end{align*}
In the latter case, we have
\begin{align*}
\|A_i\|_2\leq & \sum_{i+j=c_1}\|[M_{c_2}]_{ij}\|\leq \max\{\epsilon+1,\eta+1\}\max_{ij}\{\|[M_{c_2}]_{ij}\|_2\}\leq \\
&\max\{\epsilon+1,\eta+1\}\|M_{c_2}\|_2 \leq \max\{\epsilon+1,\eta+1\}\max_{i=0:\ell}\{|M_i\|_2\}.
\end{align*}
Hence, $\max_{i=0:d}\{\|A_i\|_2\}\leq 2\max\{\epsilon+1,\eta+1\} \max_{i=0:\ell}\{|M_i\|_2\}$, as we wanted to show.
\end{proof}

As a consequence of Theorem \ref{thm:ratio L-P} and Proposition \ref{prop:bound norm M}, a necessary, but not sufficient,  condition under which the upper bounds \eqref{eq:ratios of cond numbers2} and \eqref{eq:ratios of cond numbers} could be moderate is  the condition \begin{equation}\label{eq:scaling condition1}
\max_{i=0:d}\{\|A_i\|_2\}\approx 1.
\end{equation}
Notice that  \eqref{eq:scaling condition1} is very mild,  since it can be always achieved by dividing the original matrix polynomial by a number.

In the following section, we particularize the upper bounds \eqref{eq:ratios of cond numbers2} and \eqref{eq:ratios of cond numbers} to the case when $\mathcal{L}(\lambda)$ is a block Kronecker companion form.
We will show that  \eqref{eq:scaling condition1} is sufficient to guarantee a moderate upper bound \eqref{eq:ratios of cond numbers}.
In other words, scaling $P(\lambda)$ so that \eqref{eq:scaling condition1} is satisfied guarantees all block Kronecker companion forms to be about as well conditioned in the normwise sense as the polynomial $P(\lambda)$ itself.
Additionally, we will show that the conditions
\begin{equation}\label{eq:scaling condition2}
\max_{i=0:d}\{\|A_i\|_2\}\approx 1 \quad \quad \mbox{and} \quad \quad \min\{\|A_0\|_2,\|A_d\|_2\}\approx 1
\end{equation}
 are sufficient to guarantee a moderate upper bound \eqref{eq:ratios of cond numbers2}.
In other words, assuming the conditions in \eqref{eq:scaling condition2}, block Kronecker companion forms are  optimally conditioned in the more stringent coefficientwise sense.


\section{The conditioning of companion $\ell$-ifications}
\label{sec:conditioning2}
This section contains one of the main results of this work, Theorem \ref{thm:main1}. 
We show that block Kronecker companion forms (recall their definition given in Section \ref{sec:companion forms})  are optimally conditioned in the normwise sense if condition \eqref{eq:scaling condition1} holds, and optimally conditioned in the coefficientwise sense if the conditions in \eqref{eq:scaling condition2} hold.

Before stating the main theorems, we present Lemma \ref{lemma:bound L and M}.
\begin{lemma}\label{lemma:bound L and M}
Let $P(\lambda)$ be an $n\times n$ matrix polynomial as in \eqref{eq:poly} of degree $d$.
Assume  $d$ is divisible by $\ell$, and let $\mathcal{L}(\lambda)=\sum_{i=0}^\ell \mathcal{L}_i\lambda^i$ as in \eqref{eq:block Kronecker poly} be a block Kronecker companion form of $P(\lambda)$.
If we consider $M_t$, for $t=0,\hdots,\ell$, as an $(\eta+1)\times (\epsilon+1)$ block-matrix with $n\times n$ blocks $[M_{t}]_{ij}$, then
\begin{equation}\label{eq:cond-coeff M}
\max_{i,j,t}\{\|\,[M_{t}]_{ij}\,\|_2\}\leq\max\{1,\max_{i=0:d}\{\|A_i\|_2\}\}.
\end{equation}
\end{lemma}
\begin{proof}
The proof follows immediately from the fact that each block entry $ [M_{t}]_{ij}$ equals either $0$, $I_n$ or $A_i$, for some $i\in\{0,\hdots,d\}$.
\end{proof}


When $\mathcal{L}(\lambda)$ is a block Kronecker companion form, we can obtain upper  bounds on the ratios \eqref{eq:ratios of cond numbers2} and \eqref{eq:ratios of cond numbers} that depend essentially  on the norms of the matrix coefficients of the polynomial $P(\lambda)$.

\begin{theorem}\label{thm:ratio L-P companion}
Let $P(\lambda)$ be an $n\times n$ matrix polynomial as in \eqref{eq:poly} of degree $d$.
Assume $d$ is divisible by $\ell$, and let $\mathcal{L}(\lambda)$ as in \eqref{eq:block Kronecker poly} be a block Kronecker \emph{companion form} of $P(\lambda)$.  
If $\lambda_0$ is a simple, finite, nonzero eigenvalue of $P(\lambda)$, then
\begin{equation}\label{eq:ratio L-P 2}
\frac{\mathrm{coeff\,cond}_{L}(\lambda_0)}{\mathrm{coeff\,cond}_{P}(\lambda_0)}\leq 
16d^3(\epsilon+1)^{3/2}(\eta+1)^{3/2}\dfrac{\max\{1,\max_{i=0:d}\{\| A_i \|_2^3\} \}}{\min\{\|A_0\|_2,\|A_d\|_2\}},
\end{equation}
and
\begin{equation}\label{eq:ratio L-P 3}
\frac{\mathrm{coeff\,cond}_{L}(\lambda_0)}{\mathrm{norm\,cond}_{P}(\lambda_0)}\leq 
16d^3(\epsilon+1)^{3/2}(\eta+1)^{3/2}\frac{\max\{1,\max_{i=0:d}\{\| A_i \|_2^3\} \}}{\max_{i=0:d}\{\|A_i\|_2\}}.
\end{equation}
\end{theorem}
\begin{proof}
We only prove \eqref{eq:ratio L-P 2}.
The upper bound \eqref{eq:ratio L-P 3} follows from \eqref{eq:ratio L-P norm} using a similar argument.

First, observe that Lemmas \ref{lemma:norm bound} and \ref{lemma:bound L and M} imply
\begin{equation}\label{eq:aux0}
\|M_i\|_2\leq \sqrt{(\epsilon+1)(\eta+1)}\max\{1,\max_{i=0:d}\{\|A_i\|_2\} \},
\end{equation}
for $i=0,1,\hdots \ell$.
So, we have
\begin{equation}\label{eq:aux1}
\frac{\max\{1,\max_{i=0:\ell}\{\|M_i\|_2\}\}}{\min\{\|A_0\|_2,\|A_d\|_2\}} \leq \sqrt{(\epsilon+1)(\eta+1)}\frac{\max\{1,\max_{i=0:d}\{\|A_i\|_2\}\}}{\min\{\|A_0\|_2,\|A_d\|_2\}}.
\end{equation}
Then, using \eqref{eq:aux0}, we get the inequality
\begin{align}\label{eq:aux2}
\begin{split}
&\left( 1+\epsilon^2(\ell+1)\sum_{i=0}^\ell \|M_i\|_2^2\right)^{1/2}\leq  \\
&\left(1+\epsilon^2(\ell+1)^2(\epsilon+1)(\eta+1) \max\{1,\max_{i=0:d}\{\|A_i\|_2^2\}\} \right)^{1/2} \leq  \\ 
&\,
(\epsilon+1)^{1/2}(\eta+1)^{1/2}(1+\epsilon^2(\ell+1)^2)^{1/2}\max\{1,\max_{i=0:d}\{\|A_i\|_2\}\}\leq  \\
&\,
2d(\epsilon+1)^{1/2}(\eta+1)^{1/2}\max\{1,\max_{i=0:d}\{\|A_i\|_2\}\},
\end{split}
\end{align}
where, to get the last inequality, we have used  $\epsilon\ell \leq d$ and some elementary inequalities.
An analogous argument yields
\begin{equation}\label{eq:aux3}
\left( 1+\eta^2(\ell+1)\sum_{i=0}^\ell \|M_i\|_2^2\right)^{1/2}\leq 2d(\epsilon+1)^{1/2}(\eta+1)^{1/2}\max\{1,\max_{i=0:d}\{\|A_i\|_2\}\}.
\end{equation} 
Finally, inserting the inequalities \eqref{eq:aux1}, \eqref{eq:aux2} and \eqref{eq:aux3} in \eqref{eq:ratio L-P}, and using $(\ell+1)(\epsilon+1)^{1/2}(\eta+1)^{1/2}\leq 2d$, we obtain the desired result.
\end{proof}

\begin{remark}
The factor $16d^3(\epsilon+1)^{3/2}(\eta+1)^{3/2}$ in the upper bounds \eqref{eq:ratio L-P 2} and \eqref{eq:ratio L-P 3} may be pessimistic for some block Kronecker companion forms.
This factor takes into account the worst case scenario in which the matrices $M_i$ are dense block-matrices. 
One can obtain tighter constants, for example, particularizing the analysis to block Kronecker companion forms such that the matrices $M_i$ are of low block-bandwidth. 
In this case, the constant reduces essentially to $d^3$, result that is coherent with other analyses;  see \cite[Theorem 5.1]{tridiagonal}. 
\end{remark}

As an immediate corollary of Theorem \ref{thm:ratio L-P companion}, we obtain Theorem \ref{thm:main1}, which is one of the main results of this work.
Theorem \ref{thm:main1} gives conditions on the coefficients of $P(\lambda)$ that guarantee that \emph{all} block Kronecker companion forms of $P(\lambda)$ are about as well conditioned as the polynomial itself.
\begin{theorem}\label{thm:main1}
Let $P(\lambda)$ be an $n\times n$ matrix polynomial as in \eqref{eq:poly} of degree $d$.
Assume $d$ is divisible by $\ell$, and let $\mathcal{L}(\lambda)$ as in \eqref{eq:block Kronecker poly} be a block Kronecker companion form of $P(\lambda)$.  
If $\lambda_0$ is a simple, finite, nonzero eigenvalue of $P(\lambda)$, then the following statements hold.
\begin{itemize}
\item[\rm(a)] If $\max_{i=0:d}\{\|A_i\|_2\} =1$, then
\[
\frac{\mathrm{coeff\,cond}_{\mathcal{L}}(\lambda_0)}{\mathrm{norm\,cond}_{P}(\lambda_0)} \lesssim 1.
\]
In other words, under the scaling $\max_{i=0:d}\{\|A_i\|_2\} =1$ assumption, block Kronecker companion forms are optimally conditioned in the normwise sense.
\item[\rm(b)] If $\max_{i=0:d}\{\|A_i\|_2\} =1$ and $\min\{\|A_0\|_2,\|A_d\|_2\}=1$, then
\[
\frac{\mathrm{coeff\,cond}_{\mathcal{L}}(\lambda_0)}{\mathrm{coeff\,cond}_{P}(\lambda_0)} \lesssim 1.
\]
In other words, under the scaling $\max_{i=0:d}\{\|A_i\|_2\} =1$ assumption, block Kronecker companion forms are optimally conditioned in the coefficientwise sense provided $\min\{\|A_0\|_2,\|A_d\|_2\}$ is not too small.
\end{itemize}
\end{theorem}

The result in part (a) of Theorem \ref{thm:main1} is entirely consistent with the results in \cite{DLPV,DPV} on the backward stability of solving PEPs by using  block Kronecker companion linearizations, as we now explain.
The analyses in \cite{DLPV,DPV} show that solving a PEP by applying a backward stable algorithm to a block Kronecker companion linearization is backward stable for the PEP under the scaling condition $\|\begin{bmatrix} A_0 & \cdots & A_d \end{bmatrix} \|_F \approx 1$. 
This means that the computed eigenvalues are the exact eigenvalues of a perturbed matrix polynomial
\[
P(\lambda)+\Delta P(\lambda) = \sum_{i=0}^d(A_i+\Delta A_i)\lambda^i
\]
where $\|\begin{bmatrix}\Delta A_0 & \cdots & \Delta A_d \end{bmatrix}\|_F = \mathcal{O}(u)\|\begin{bmatrix} A_0 & \cdots & A_d \end{bmatrix} \|_F$. 
Hence, the eigenvalue relative errors can be bounded as
\[
\frac{|\lambda_i-\widetilde{\lambda_i}|}{|\lambda_i|}=\mathcal{O}(u\cdot \mathrm{norm\,cond}_P(\lambda_i)), 
\]
where $\lambda_i$ are the exact eigenvalues of $P(\lambda)$ and $\widetilde{\lambda_i}$ are the computed eigenvalues.
In conclusion, \cite{DLPV,DPV} show that well-conditioned eigenvalues in normwise sense (i.e., $\mathrm{norm\,cond}_P(\lambda)\approx 1$) can be computed with high relative accuracy if we apply a backward stable algorithm to the block Kronecker companion linearization.
The same conclusion can be drawn from  part (a) in Theorem \ref{thm:main1} .

\section{Comparing the coefficientwise conditioning of different block Kronecker companion forms}
\label{sec:conditioning3}

In the coefficientwise sense, even after scaling the matrix polynomial $P(\lambda)$ so that $\max_{i=0:d}\{\|A_i\|_2\}=1$, block Kronecker companion forms may be potentially much worse conditioned than the polynomial $P(\lambda)$ when the quantity $\min\{\|A_0\|_2,\|A_d\|_2\}$ is much smaller than one.
For this reason, we investigate in this section whether  some block Kronecer companion forms are preferable to others, from an eigenvalue conditioning point of view.

\begin{theorem}\label{thm:ratio different ell-ifications}
Let $P(\lambda)$ be an $n\times n$ matrix polynomial as in \eqref{eq:poly} of degree $d$.
Assume $d$ is divisible both by $\ell$ and $r$.
Let $\mathcal{L}(\lambda)=\sum_{i=0}^\ell \mathcal{L}_i\lambda^i$  be an $(\epsilon_1,n,\eta_1,n)$-block Kronecker companion $\ell$-ification of $P(\lambda)$, and let $\mathcal{R}(\lambda)=\sum_{i=0}^r\mathcal{R}_i\lambda^i$ be an $(\epsilon_2,n,\eta_2,n)$-block Kronecker companion $r$-ification of $P(\lambda)$.
If $\lambda_0$ is a finite, nonzero and simple eigenvalue of $P(\lambda)$, then
\begin{align*}
\frac{1}{16d^3(\epsilon_1+1)^{3/2}(\eta_1+1)^{3/2}\max\{1,\max_{i=0:d}\{\|A_i\|_2^3\}}&\leq \\ 
&\hspace{-5.2cm}\frac{\mathrm{coeff\,cond}_\mathcal{R}(\lambda_0)}{\mathrm{coeff\,cond}_\mathcal{L}(\lambda_0)} \leq 16d^3(\epsilon_2+1)^{3/2}(\eta_2+1)^{3/2}\max\{1,\max_{i=0:d}\{\|A_i\|_2^3\}\}.
\end{align*}
\end{theorem}
\begin{proof}
Throughout the proof, we assume $\ell\geq r$.
The case when $\ell<r$ is completely analogous and left to the reader.

Let $(y,\lambda_0,x)$, $(w_\mathcal{L},\lambda_0,z_\mathcal{L})$ and $(w_\mathcal{R},\lambda_0,z_\mathcal{R})$ be eigentriples of, respectively, the polynomial $P(\lambda)$, the $\ell$-ification $\mathcal{L}(\lambda)$ and the $r$-ification $\mathcal{R}(\lambda)$.
We have
\begin{equation}\label{eq:aux4}
\dfrac{\mathrm{coeff\,cond}_\mathcal{R}(\lambda_0)}{\mathrm{coeff\,cond}_\mathcal{L}(\lambda_0)}=\dfrac{\sum_{i=0}^r|\lambda_0|^i\|\mathcal{R}_i\|_2}{\sum_{i=0}^\ell|\lambda_0|^i\|\mathcal{L}_i\|_2}\cdot \dfrac{\|z_\mathcal{R}\|_2\|w_\mathcal{R}\|_2}{\|z_\mathcal{L}\|_2\|w_\mathcal{L}\|_2}\cdot\dfrac{|w^*_\mathcal{L}\mathcal{L}^\prime(\lambda_0)z_\mathcal{L}|}{|w^*_\mathcal{R}\mathcal{R}^\prime(\lambda_0)z_\mathcal{R}|}.
\end{equation}
To upper bound the ratio \eqref{eq:aux4}, we need to distinguish two cases, namely, the cases $|\lambda_0|>1$ and $|\lambda_0|\leq 1$.

Assume first $|\lambda_0|>1$. 
Notice that Lemma \ref{lemma:norm bound} implies $\|\mathcal{L}_\ell\|_2\geq 1$, because the leading matrix coefficient $\mathcal{L}_\ell$ has at least one block entry equal to $I_n$. 
If we denote by $N(\lambda)=\sum_{i=0}^r N_i\lambda^i$ the (1,1) block of $\mathcal{R}(\lambda)$, then, from Lemmas \ref{lemma:norm bound} and \ref{lemma:bound L and M}, we obtain
\begin{align}\label{eq:aux5}
\begin{split}
\dfrac{\sum_{i=0}^r|\lambda_0|^i\|\mathcal{R}_i\|_2}{\sum_{i=0}^\ell|\lambda_0|^i\|\mathcal{L}_i\|_2} \leq & \dfrac{\sum_{i=0}^r|\lambda_0|^i\|\mathcal{R}_i\|_2}{|\lambda_0|^\ell\|\mathcal{L}_\ell\|_2} \leq \dfrac{1}{|\lambda_0|^{\ell-r}}\sum_{i=0}^r|\lambda_0|^{i-r}\|\mathcal{R}_i\|_2 \leq \\
&\frac{1}{|\lambda_0|^{\ell-r}}(r+1)\max_{i=0:r}\{\|\mathcal{R}_i\|_2\} \leq \\ 
&\frac{2}{|\lambda_0|^{\ell-r}}(r+1)\max\{1,\max_{i=0:r}\{\|N_i\|_2\} \}\leq  \\
&\dfrac{2}{|\lambda_0|^{\ell-r}}\sqrt{(\epsilon_2+1)(\eta_2+1)}\max\{1,\max_{i=0:d}\{\|A_i\|_2\}\}.
\end{split}
\end{align}
Next, recall the definition of the matrix polynomial $G(\lambda;p,q,M)$ introduced in \eqref{eq:G}, and let us denote by $M(\lambda)$ the $(1,1)$ block of $\mathcal{L}(\lambda)$.
From part (b) of Theorems \ref{thm:right eigenvectors} and \ref{thm:left eigenvectors}, we have
\begin{align}\label{eq:aux6}
\begin{split}
\dfrac{\|z_\mathcal{R}\|_2\|w_\mathcal{R}\|_2}{\|z_\mathcal{L}\|_2\|w_\mathcal{L}\|_2} = &
\dfrac{\|G(\lambda_0;\epsilon_2,\eta_2,N)x\|_2\|G(\overline{\lambda_0};\eta_2,\epsilon_2,N^*)y\|_2}{\|G(\lambda_0;\epsilon_1,\eta_1,M)x\|_2\|G(\overline{\lambda_0};\eta_1,\epsilon_1,M^*)y\|_2}\leq  \\
& \dfrac{\|G(\lambda_0;\epsilon_2,\eta_2,N)x\|_2\|G(\overline{\lambda_0};\eta_2,\epsilon_2,N^*)y\|_2}{|\lambda_0|^{2(d-\ell)}\|x\|_2\|y\|_2} =  \\
&|\lambda_0|^{2(\ell-r)}\cdot\dfrac{\|G(\lambda_0;\epsilon_2,\eta_2,N)x\|_2}{|\lambda_0|^{d-r}\|x\|_2}\dfrac{\|G(\overline{\lambda_0};\eta_2,\epsilon_2,N^*)y\|_2}{|\lambda_0|^{d-r}\|y\|_2} \leq \\
&|\lambda_0|^{2(\ell-r)}\cdot 4d^2(\epsilon_2+1)^{3/2}(\eta_2+1)^{3/2}\max\{1,\max_{i=0:d}\{\|A_i\|_2^2\}\},
\end{split}
\end{align}
where the first inequality  above follows from  $\|G(\lambda_0;\epsilon_1,\eta_1,M)x\|_2\geq |\lambda_0|^{d-\ell}\|x\|_2$ and $\|G(\overline{\lambda_0};\eta_1,\epsilon_1,M^*)y\|_2\geq |\lambda_0|^{d-\ell}\|y\|_2$, and the second inequality follows from  similar arguments to the ones used  in the proofs of \eqref{eq:X1} and \eqref{eq:X2}.
Finally, observe that part (b) in Lemma \ref{lemma:relating cond numb}, implies
\begin{equation}\label{eq:aux7}
\dfrac{|w^*_\mathcal{L}\mathcal{L}^\prime(\lambda_0)z_\mathcal{L}|}{|w^*_\mathcal{R}\mathcal{R}^\prime(\lambda_0)z_\mathcal{R}|} = \dfrac{|\lambda_0|^{d-\ell}|y^*P^\prime(\lambda_0)x|}{|\lambda|^{d-r}|y^*P^\prime(\lambda_0)x|}=\dfrac{1}{|\lambda_0|^{\ell-r}}.
\end{equation}
Using the inequalities \eqref{eq:aux5}--\eqref{eq:aux7}, we obtain from \eqref{eq:aux4}
\begin{align*}
\dfrac{\mathrm{coeff\,cond}_\mathcal{R}(\lambda_0)}{\mathrm{coeff\,cond}_\mathcal{L}(\lambda_0)} \leq 8d^2(r+1)(\epsilon_2+1)^{2}(\eta_2+1)^{2} \max\{1,\max_{i=0:d}\{\|A_i\|_2^3\}\} \leq &\\
16d^3(\epsilon_2+1)^{3/2}(\eta_2+1)^{3/2}\max\{1,\max_{i=0:d}\{\|A_i\|_2^3\}\} &,
\end{align*}
which is the desired upper bound.

Assume now $|\lambda_0|\leq 1$. 
Observe that Lemma \ref{lemma:norm bound} implies $\|\mathcal{L}_0\|_2\geq 1$, since the trailing matrix coefficient $\mathcal{L}_0$ has at least one block entry equal to $I_n$. 
Then, from Lemmas \ref{lemma:norm bound} and \ref{lemma:bound L and M}, we easily obtain
\begin{align}\label{eq:aux8}
\begin{split}
\dfrac{\sum_{i=0}^r|\lambda_0|^i\|\mathcal{R}_i\|_2}{\sum_{i=0}^\ell|\lambda_0|^i\|\mathcal{L}_i\|_2} \leq  \dfrac{\sum_{i=0}^r|\lambda_0|^i\|\mathcal{R}_i\|_2}{\|\mathcal{L}_0\|_2} \leq \sum_{i=0}^r|\lambda_0|^{i}\|\mathcal{R}_i\|_2 \leq 
(r+1)\max_{i=0:r}\{\|\mathcal{R}_i\|_2\} \leq &\\
(r+1)\sqrt{(\epsilon_2+1)(\eta_2+1)}\max\{1,\max_{i=0:d}\{\|A_i\|_2\}\}.&
\end{split}
\end{align}
Next, from part (a) in Theorems \ref{thm:right eigenvectors} and \ref{thm:left eigenvectors}, we have
\begin{align}\label{eq:aux9}
\begin{split}
\dfrac{\|z_\mathcal{R}\|_2\|w_\mathcal{R}\|_2}{\|z_\mathcal{L}\|_2\|w_\mathcal{L}\|_2} = &
\dfrac{\|H(\lambda_0;\epsilon_2,\eta_2,N)x\|_2\|H(\overline{\lambda_0};\eta_2,\epsilon_2,N^*)y\|_2}{\|H(\lambda_0;\epsilon_1,\eta_1,M)x\|_2\|H(\overline{\lambda_0};\eta_1,\epsilon_1,M^*)y\|_2}\leq  \\
&\dfrac{\|H(\lambda_0;\epsilon_2,\eta_2,N)x\|_2}{\|x\|_2}\dfrac{\|H(\overline{\lambda_0};\eta_2,\epsilon_2,N^*)y\|_2}{\|y\|_2} \leq \\
&4d^2(\epsilon_2+1)^{3/2}(\eta_2+1)^{3/2}\max\{1,\max_{i=0:d}\{\|A_i\|_2^2\}\},
\end{split}
\end{align}
where the first inequality follows from the inequalities $\|H(\lambda;\epsilon_1,\eta_1,M)x\|_2\geq \|x\|_2$ and $\|H(\overline{\lambda_0};\eta_1,\epsilon_1,M^*)y\|_2\geq \|y\|_2$, and the second inequality follows from  similar arguments to the ones used  in the proofs of \eqref{eq:Y1} and \eqref{eq:Y2}.
Finally, we obtain from Lemma \ref{lemma:relating cond numb}
\begin{equation}\label{eq:aux10}
\dfrac{|w^*_\mathcal{L}\mathcal{L}^\prime(\lambda_0)z_\mathcal{L}|}{|w^*_\mathcal{R}\mathcal{R}^\prime(\lambda_0)z_\mathcal{R}|} = \dfrac{|y^*P^\prime(\lambda_0)x|}{|y^*P^\prime(\lambda_0)x|}=1.
\end{equation}
Using the inequalities \eqref{eq:aux8}---\eqref{eq:aux10} to bound \eqref{eq:aux4}, the desired upper bound readily follows.

We finally observe that the lower  bound follows from applying the just established upper bound to  $\mathrm{coeff\,cond}_\mathcal{L}(\lambda_0)/\mathrm{coeff\,cond}_\mathcal{R}(\lambda_0)$.
\end{proof}

As an immediate corollary of Theorem \ref{thm:ratio different ell-ifications}, we obtain the second main result of this work.
From the conditioning point of view, Theorem \ref{thm:main2} establishes that no block Kronecker companion form is more preferable to other block Kronecker companion form, provided we scale the polynomial $P(\lambda)$ so that  $\max_{i=0:d}\{\|A_i\|_2\}=1$.

\begin{theorem}\label{thm:main2}
Let $P(\lambda)$ be an $n\times n$ matrix polynomial as in \eqref{eq:poly} of degree $d$.
Assume $d$ is divisible both by $\ell$ and $r$, and let $\mathcal{L}(\lambda)$ be a block Kronecker companion $\ell$-ification of $P(\lambda)$, and let $\mathcal{R}(\lambda)$ be a block Kronecker companion $r$-ification of $P(\lambda)$.
Assume that $P(\lambda)$ has been scaled so that $\max_{i=0:d}\{\|A_i\|_2\}=1$.
If $\lambda_0$ is a finite, nonzero and simple eigenvalue of $P(\lambda)$, then
\[
\frac{\mathrm{coeff\,cond}_\mathcal{R}(\lambda_0)}{\mathrm{coeff\,cond}_\mathcal{L}(\lambda_0)}\approx 1.
\]
In other words, under the scaling $\max_{i=0:d}\{\|A_i\|_2\}=1$ assumption, no block Kronecker companion form is much better or much worse conditioned than other block Kronecker companion form.
\end{theorem}

\section{Numerical examples}\label{sec:experiments}

We illustrate the theory on some random and on benchmark matrix polynomials from the NLEVP collection \cite{NLEVP}.
Our experiments were performed in MATLAB 8, for which the unit
roundoff is $2^{-53} \approx 10^{-16}$.
To obtain condition numbers, we took as exact eigenvalues and eigenvectors  the ones computed in MATLAB's VPA arithmetic at 40 digit precision (except in Section \ref{sec:plasma}, which was not possible due to the large size of the problem). 
The $x$-axis in all our figures represents eigenvalue index. The eigenvalues are always sorted in increasing order of absolute value.

\subsection{Experiment 1: well-conditioned eigenvalues}
One of the main predictions of our theory is that well-conditioned eigenvalues can be computed with high relative accuracy as the eigenvalues of block Kronecker companion forms.
The goal of this experiment is to verify this prediction.
Due to the lack of software for computing eigenvalues of low, but larger than one, degree  matrix polynomials, we will focus only on companion linearizations.

We generate a random $n\times n$ matrix polynomial $P(\lambda)$ as in \eqref{eq:poly} with degree $d=3$ and size $n=30$.
We construct each matrix coefficient $A_i$ with the MATLAB line command
\[
{\tt randn(n)+sqrt(-1)*randn(n)};
\]
Then, we compute the eigenvalues of $P(\lambda)$ as the eigenvalues of the Frobenious companion form 
\begin{equation}\label{eq:exp2 Frobenius}
\mathcal{L}_1(\lambda)=\begin{bmatrix}
\lambda A_3+A_2 & A_1 & A_0 \\ -I_n &\lambda I_n & 0 \\ 0 & -I_n & \lambda I_n
\end{bmatrix},
\end{equation}
a (permuted) Fiedler pencil $L_2(\lambda)$, a (permuted) generalized Fiedler pencil $L_3(\lambda)$, and another block Kronecker pencil $L_4(\lambda)$, where
\begin{align}\label{eq:exp1 linearizations}
\begin{split}
&
\mathcal{L}_2(\lambda)=\begin{bmatrix}
\lambda A_3+A_2 & A_1 & -I_n \\ 0 & A_0 & \lambda I_n \\ -I_n & \lambda I_n & 0
\end{bmatrix},\\
& \mathcal{L}_3(\lambda)=\begin{bmatrix}
\lambda A_3+A_2 & 0 & -I_n \\ 
0 & \lambda A_1+A_0 & \lambda I_n \\
-I_n & \lambda I_n & 0
\end{bmatrix}, \quad \quad \mbox{and} \\ 
&\mathcal{L}_4(\lambda)=\begin{bmatrix}
\lambda A_3-A_2 & \lambda A_2+A_1 & -I_n \\
\lambda A_2+A_1 & -\lambda A_1+A_0 & \lambda I_n \\
-I_n & \lambda I_n & 0
\end{bmatrix}
\end{split}
\end{align}
Observe that the four pencils in \eqref{eq:exp2 Frobenius}-\eqref{eq:exp1 linearizations} are block Kronecker companion forms of the matrix polynomial $P(\lambda)=\sum_{i=0}^4 A_i\lambda^i$.

All the eigenvalues of the matrix polynomial $P(\lambda)$ are well-conditioned in the normwise sense \cite{random-poly}.
Since $\max_{i=0:3}\{\|A_i\|_2\}$ is approximately equal to 1, our theory predicts that all the eigenvalues of $P(\lambda)$ must be computed with high relative accuracy, regardless of the block Kronecker companion employed.
This is confirmed in Figure \ref{fig:forward errors}, where we plot the relative forward errors
\begin{equation}\label{eq:exp1 forward}
\dfrac{|\lambda_i-\widetilde{\lambda}_i|}{|\lambda_i|} \quad \quad \lambda_i\mbox{: exact eigenvalue,} \quad  \quad \widetilde{\lambda}_i\mbox{: computed eigenvalue},
\end{equation}
for $i=1,2,\hdots,90$, for the the four linearizations in \eqref{eq:exp2 Frobenius}-\eqref{eq:exp1 linearizations}. 
\begin{figure}[h]
\centering
\includegraphics[height=6.5cm, width=13cm]{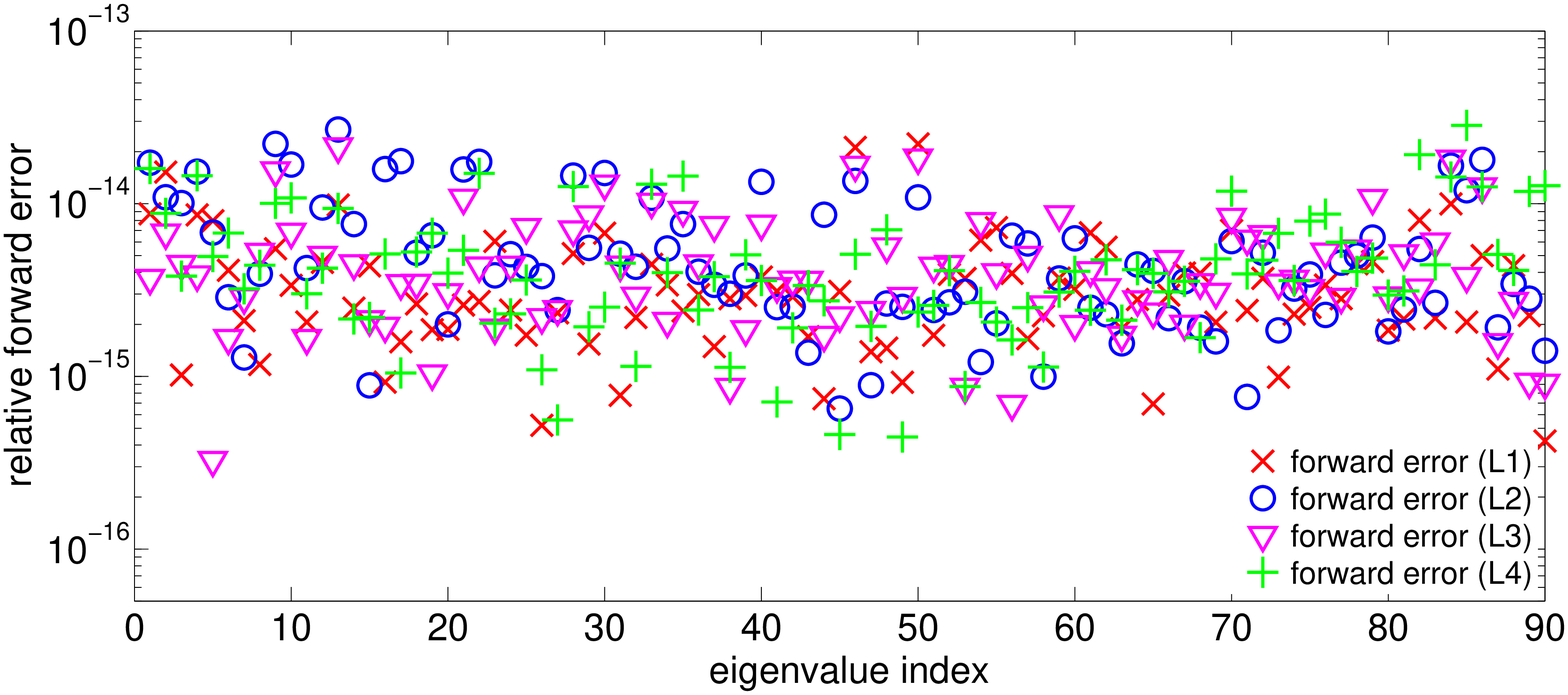}
\caption{Relative forward errors \eqref{eq:exp1 forward} of the computed eigenvalues of a random matrix polynomial $P(\lambda)$.
The eigenvalues of $P(\lambda)$ were computed as the eigenvalues of the following block Kronecker companion forms: a Frobenius companion form $\mathcal{L}_1(\lambda)$, a (permuted) Fiedler pencil $\mathcal{L}_2(\lambda)$, a (permuted) generalized Fiedler pencil $\mathcal{L}_3(\lambda)$, and a block Kronecker pencil $\mathcal{L}_4(\lambda)$.
The pencils $\mathcal{L}_1(\lambda)$, $\mathcal{L}_2(\lambda)$, $\mathcal{L}_3(\lambda)$ and $\mathcal{L}_4(\lambda)$ are as in \eqref{eq:exp2 Frobenius}-\eqref{eq:exp1 linearizations}. 
Observe that all the eigenvalues are computed with high relative accuracy, as predicted by our theory.} 
\label{fig:forward errors} 
\end{figure}

\subsection{Experiment 2: the condition of block Kronecker $\ell$-ifications relative to that of the Frobenius companion form}\label{sec:plasma}
Another key prediction of our theory is that under the scaling $\max_{i=0:d}\{\|A_i\|_2\}=1$ assumption the Frobenius  companion forms \eqref{eq:C1} and \eqref{eq:C2} are not better (or worse) conditioned than any other block Kronecker companion form.
The goal of the following three experiments is to verify this prediction.

In the first experiment, we will compare the conditioning of the block Kronecker companion forms $\mathcal{L}_2(\lambda)$, $\mathcal{L}_3(\lambda)$, $\mathcal{L}_4(\lambda)$ in \eqref{eq:exp1 linearizations} with that of the Frobenius companion form  \eqref{eq:exp2 Frobenius}.
We consider the ``plasma drift'' matrix polynomial from the NLEVP collection \cite{NLEVP}. 
This is a matrix polynomial with degree $d=3$, size $n=128$ and $\max_{i=0:d}\{\|A_i\|_2\}\approx 1.2\times 10^3$.
In Figure \ref{fig:plasma_drift1}, we plot the ratios
\begin{equation}\label{eq:exp2 ratio}
\frac{\mathrm{coeff\,cond}_{\mathcal{L}_i}(\lambda)}{\mathrm{coeff\,cond}_{\mathcal{L}_1}(\lambda)}\quad \mbox{for }i=2,3,4.
\end{equation}
for the scaled polynomial (lower figure) and the unscaled polynomial (upper figure). 
Notice that the results in Figure \ref{fig:plasma_drift1} are in complete accordance with our theory:
a potentially large or small ratios in the unscaled case, and ratios approximately equal to 1 in the scaled case. 

\begin{figure}[h]
\centering
\includegraphics[height=6.5cm, width=13cm]{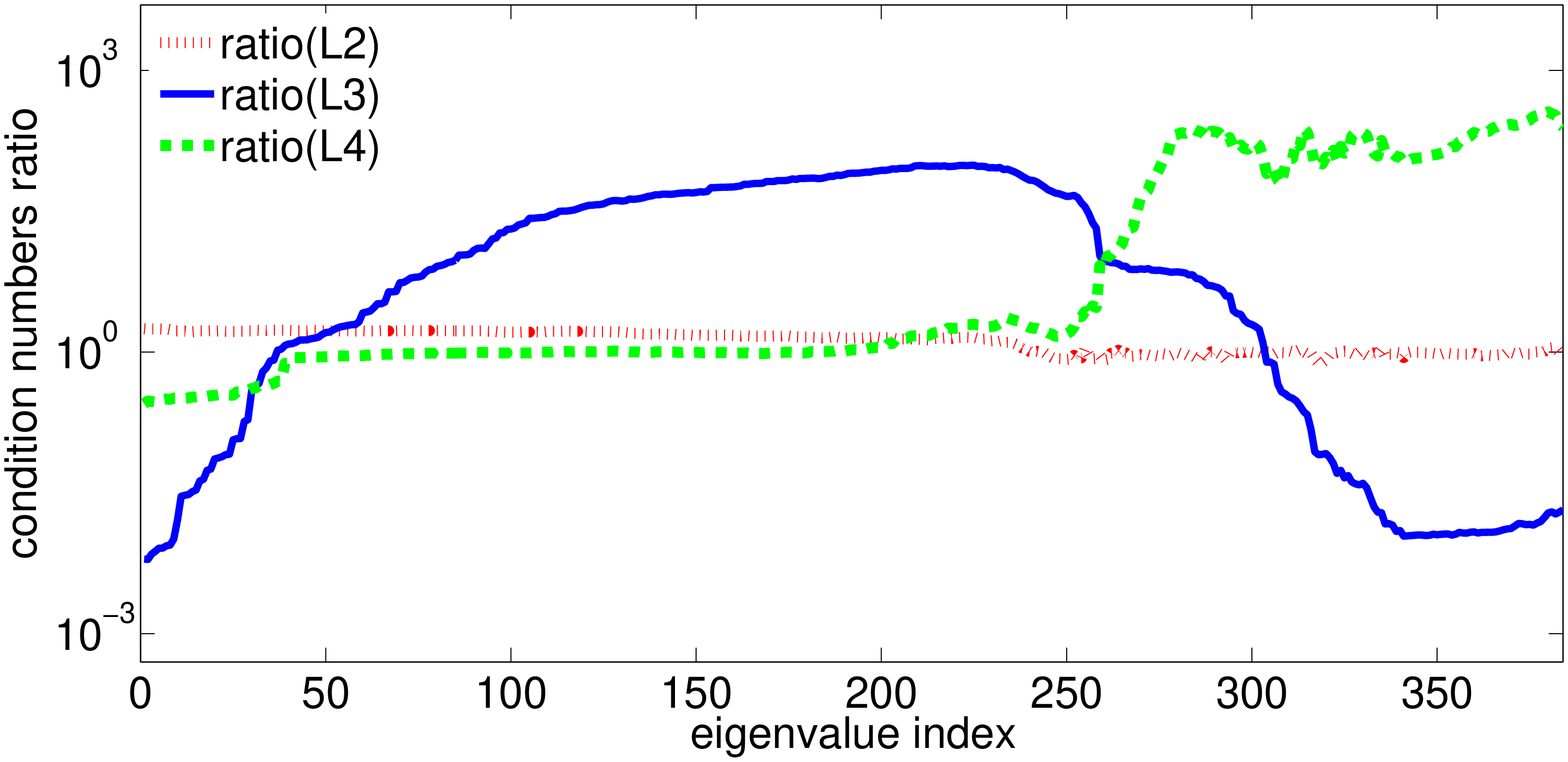}\\
\includegraphics[height=6.5cm, width=13cm]{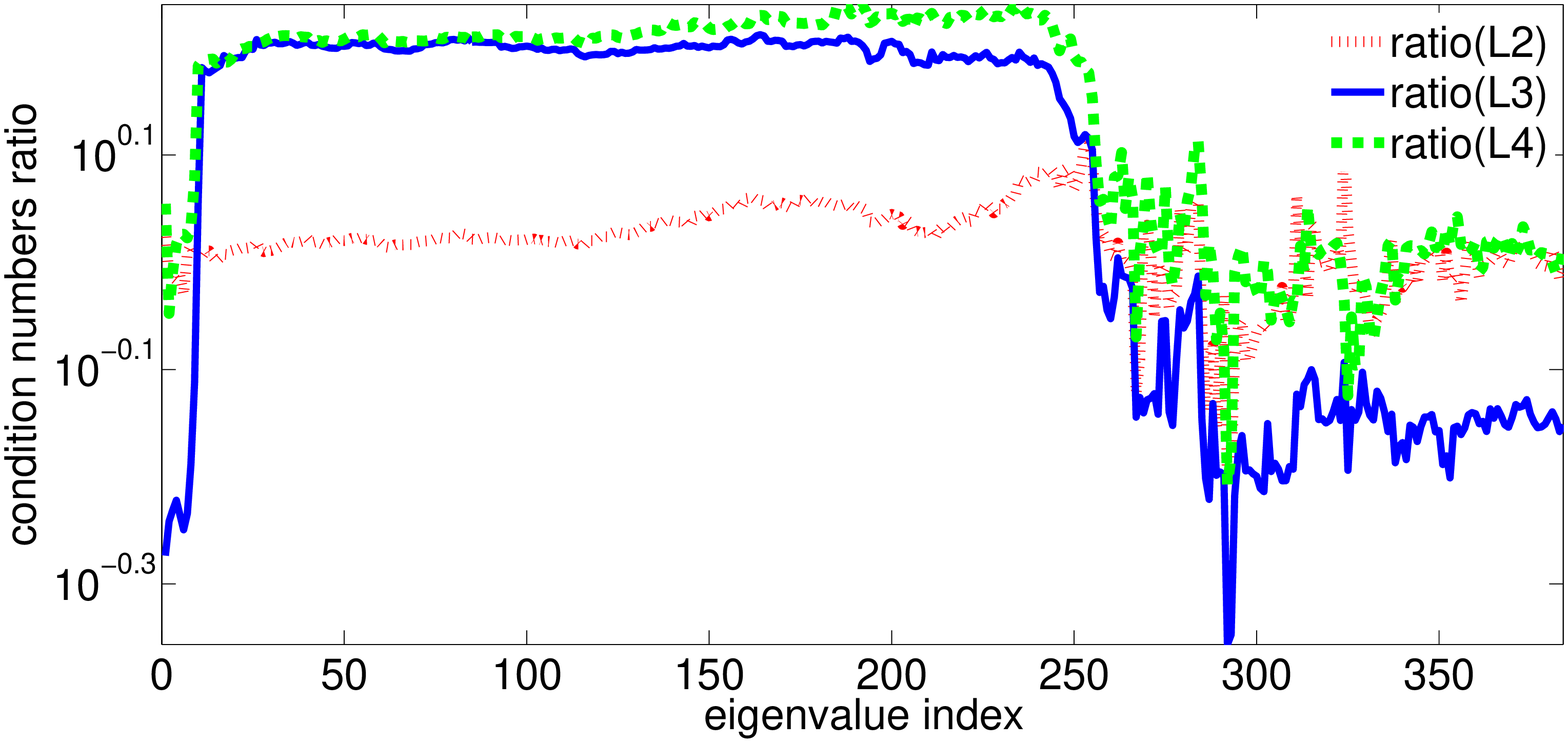}
\caption{Condition numbers ratios \eqref{eq:exp2 ratio} for the unscaled (upper figure) and scaled (lower figure) ``plasma drift'' matrix polynomial.
The pencils $\mathcal{L}_2(\lambda)$, $\mathcal{L}_3(\lambda)$ and $\mathcal{L}_4(\lambda)$ are as in \eqref{eq:exp1 linearizations}.} 
\label{fig:plasma_drift1} 
\end{figure}

In the second experiment, we will compare the conditioning of the block Kronecker quadratification (Frobenius-like quadratification) 
\begin{equation}\label{eq:exp2-Q}
Q(\lambda)=\begin{bmatrix}
\lambda^2 A_4 + \lambda A_3 & \lambda^2A_2+\lambda A_1+A_0 \\
-I_n & \lambda^2 I_n
\end{bmatrix},
\end{equation}
and the Frobenius companion form
\begin{equation}\label{eq:exp2-L}
\mathcal{L}(\lambda)=\begin{bmatrix}
\lambda A_4+A_3 & A_2 & A_1 & A_0 \\
-I_n & \lambda I_n & 0 & 0 \\
0 & -I_n & \lambda I_n & 0 \\
0 & 0 & -I_n & \lambda I_n 
\end{bmatrix},
\end{equation}
both associated with a matrix polynomial of degree 4.
We will consider the ``Orr-Sommerfeld'' matrix polynomial from the NLEVP collection \cite{NLEVP}.
This polynomial has degree $d=4$, size $n=64$, and $\max_{i=0:d}\{\|A_i\|_2\}\approx 2^{12}$.
In Figure \ref{fig:orr}, we plot the ratio
\begin{equation}\label{eq:exp2 ratio2}
\frac{\mathrm{coeff\,cond}_{\mathcal{L}}(\lambda_0)}{\mathrm{coeff\,cond}_{\mathcal{Q}}(\lambda_0)},
\end{equation}
for the scaled polynomial (red dashed line) and the unscaled polynomial (blue solid line). 
Figure \ref{fig:orr} confirms the prediction of our theory: scaling the polynomial guarantees the ratio \eqref{eq:exp2 ratio2} to be moderate. 
\begin{figure}[h]
\centering
\includegraphics[height=6.5cm, width=13cm]{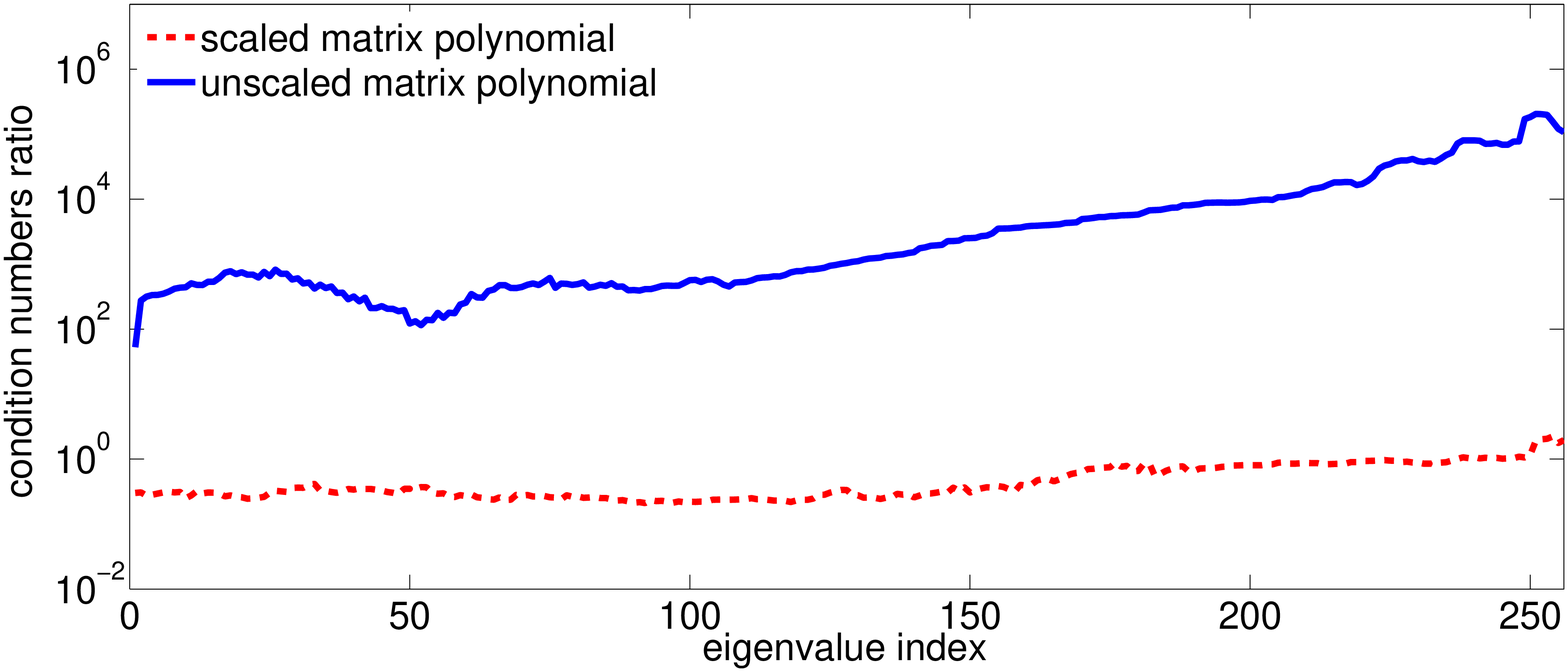}
\caption{Condition numbers ratio \eqref{eq:exp2 ratio2}  for the unscaled (blue solid line) and scaled (red dashed line) ``Orr-Sommerfeld'' matrix polynomial.} 
\label{fig:orr} 
\end{figure}

In the last experiment of this section, we consider a random matrix polynomial $P(\lambda)$ as in \eqref{eq:poly} with degree $d=6$, size $n=10$, and with badly-scaled matrix coefficients.
This matrix polynomial is constructed as follows:
\begin{align*}
&A_0 = {\tt randn(n)+sqrt(-1)*randn(n);} \\
& A_1 = {\tt 1e3*(randn(n)+sqrt(-1)*randn(n));} \\ &A_2 = {\tt randn(n)+sqrt(-1)*randn(n);} \\
&A_3 = {\tt 1e4*(randn(n)+sqrt(-1)*randn(n));} \\
&A_4 = {\tt 1e4*(randn(n)+sqrt(-1)*randn(n));} \\ &A_5  = {\tt 1e2*(randn(n)+sqrt(-1)*randn(n));} \\
&A_6 = {\tt randn(n)+sqrt(-1)*randn(n);}
\end{align*}
In the experiment, we study the conditioning of the following three block Kronecker companion $\ell$-ifications: the linearization (1-ification)
\begin{equation}\label{eq:exp2bis-F}
\mathcal{F}(\lambda) = \left[ \begin{array}{cccccc}
\lambda A_6 & \lambda A_5 & \lambda A_4 & -I_n & 0 & 0 \\
0 & 0 & \lambda A_3 & \lambda I_n & -I_n & 0 \\
0 & 0 & \lambda A_2 & 0 & \lambda I_n & -I_n \\
0 & 0 & \lambda A_1+A_0 & 0 & 0 & \lambda I_n \\ 
-I_n & \lambda I_n & 0 & 0 & 0 & 0 \\
0 & -I_n & \lambda I_n & 0 & 0 & 0
\end{array}\right],
\end{equation}
the quadratification (2-ification)
\begin{equation}\label{eq:exp2bis-Q}
\mathcal{Q}(\lambda) = \left[\begin{array}{ccc}
\lambda^2 A_6+\lambda A_5 & A_2 & -I_n \\
\lambda^2 A_4+\lambda A_3 & \lambda A_1 + A_0 & \lambda^2 I_n \\ 
-I_n & \lambda^2 I_n & 0
\end{array}\right],
\end{equation}
and the cubification (3-ification)
\begin{equation}\label{eq:exp2bis-C}
\mathcal{C}(\lambda) = \left[\begin{array}{cc}
\lambda^3 A_6+\lambda^2A_5+\lambda A_4 & \lambda^3A_3+\lambda^2A_2+\lambda A_1+A_0 \\  -I_n & \lambda^3 I_n
\end{array}\right],
\end{equation}
relative to the conditioning of the Frobenius companion form
\begin{equation}\label{eq:exp2bis-L}
\mathcal{L}(\lambda) =
\begin{bmatrix}
\lambda A_6+A_5 & A_4 & A_3 & A_2 & A_1 & A_0 \\
-I_n & \lambda I_n & 0 & 0 & 0 & 0 \\
0 & -I_n & \lambda I_n & 0 & 0 & 0 \\
0 & 0 & -I_n & \lambda I_n & 0 & 0 \\
0 & 0 & 0 & -I_n & \lambda I_n & 0 \\
0 & 0 & 0 & 0 & -I_n & \lambda I_n 
\end{bmatrix}.
\end{equation}
In Figure \ref{fig:ell-ifications}, we plot the ratios
\begin{equation}\label{eq:exp2 ratio3}
\frac{\mathrm{coeff\,cond}_{\mathcal{F}}(\lambda_0)}{\mathrm{coeff\,cond}_{\mathcal{L}}(\lambda_0)}, \quad \quad \frac{\mathrm{coeff\,cond}_{\mathcal{Q}}(\lambda_0)}{\mathrm{coeff\,cond}_{\mathcal{L}}(\lambda_0)} \quad \quad \mbox{and} \quad \quad \frac{\mathrm{coeff\,cond}_{\mathcal{C}}(\lambda_0)}{\mathrm{coeff\,cond}_{\mathcal{L}}(\lambda_0)},
\end{equation}
for the scaled polynomial (lower figure) and the unscaled polynomial (upper figure). 
Figure \ref{fig:orr} validates our theory: scaling the polynomial makes the ratios \eqref{eq:exp2 ratio3} close to 1, regardless of how badly-scaled is the polynomial (i.e., regardless of how small is the quantity $\min\{\|A_0\|_2,\|A_d\|_2\}$). 

\begin{figure}[h]
\centering
\includegraphics[height=6.5cm, width=13cm]{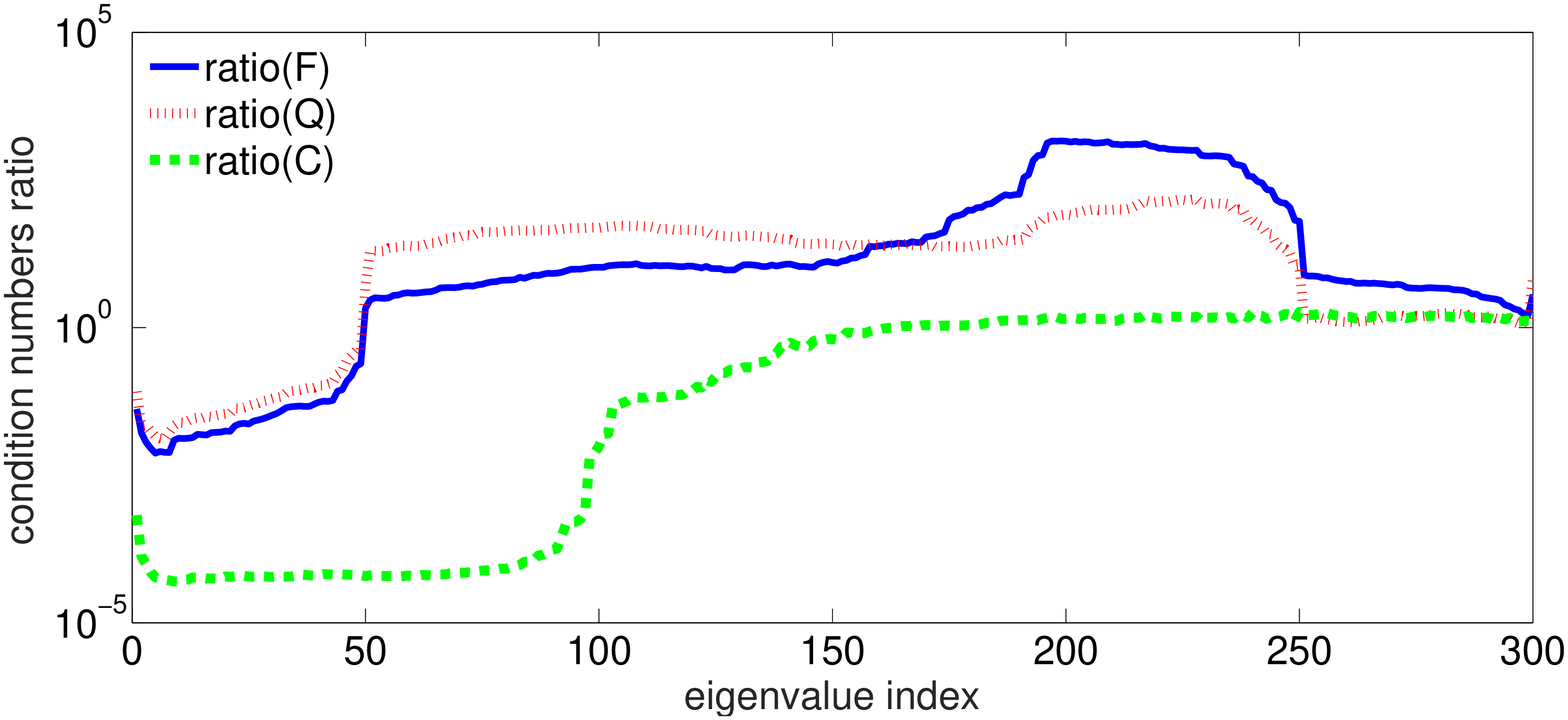}\\
\includegraphics[height=6.5cm, width=13cm]{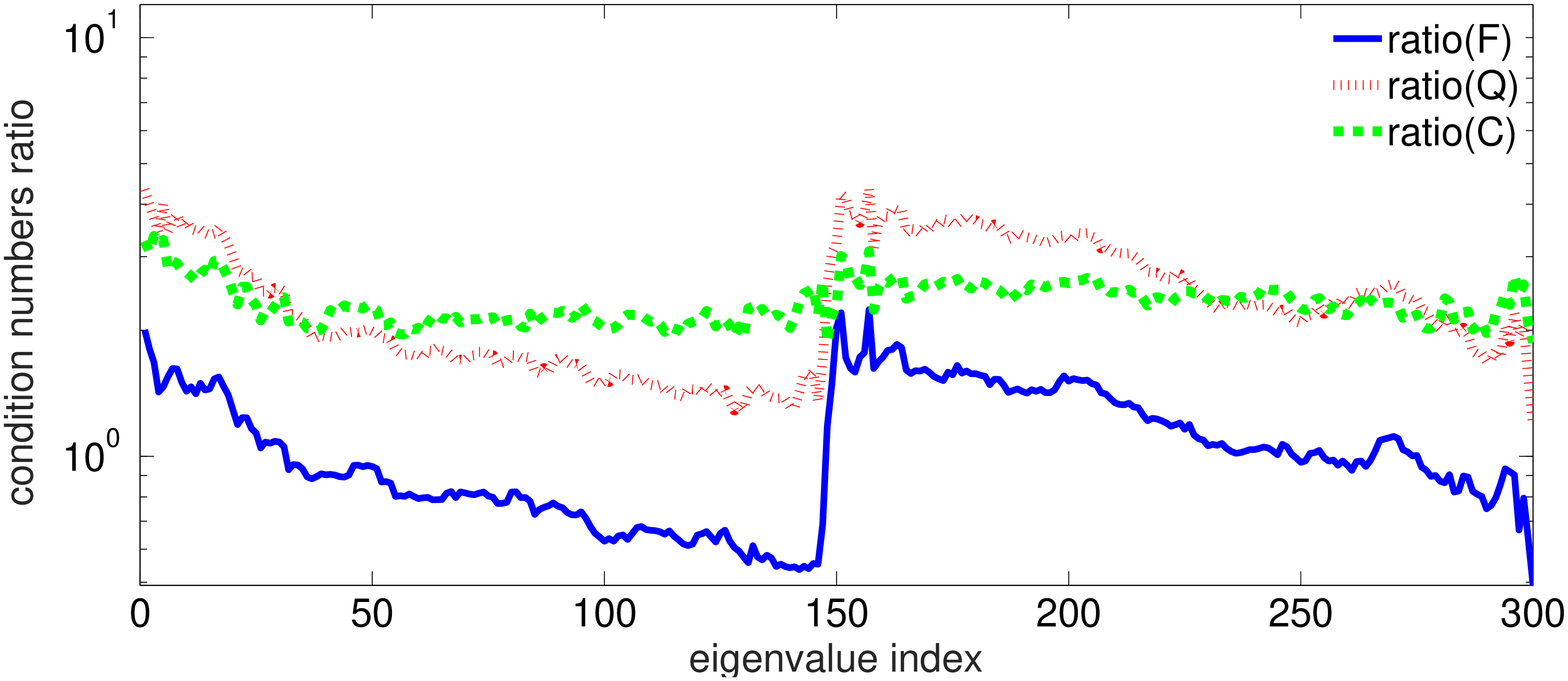}
\caption{Condition numbers ratios \eqref{eq:exp2 ratio3} for a random matrix polynomial with badly scaled matrix coefficients.
Results for the unscaled polynomial are in the upper figure, and results for the scaled polynomial are in the lower figure. The matrix polynomials $\mathcal{F}(\lambda)$, $\mathcal{Q}(\lambda)$ and $\mathcal{C}(\lambda)$ are as in \eqref{eq:exp2bis-F}--\eqref{eq:exp2bis-C}.} 
\label{fig:ell-ifications} 
\end{figure}

\subsection{Experiment 3: the conditioning of block Kronecker companion forms relative to that of the polynomial}

Ideally, we would like the block Kronecker companion form $\mathcal{L}(\lambda)$ that we use to solve a PEP to be as well conditioned as the original polynomial $P(\lambda)$.
Our theory predicts  that the coefficientwise conditioning of $\mathcal{L}(\lambda)$ is within a factor
\begin{equation}\label{eq:rho factor}
\rho(P):=\frac{\max_{i=0:d}\{\|A_i\|_2^3\}}{\min\{\|A_0\|_2,\|A_d\|_2\}}
\end{equation}
of the coefficientwise conditioning of $P(\lambda)$.
Hence, if we scale the matrix polynomial so that $\max_{i=0:d}\{\|A_i\|_2\}=1$, then $\mathcal{L}(\lambda)$ and $P(\lambda)$ are guaranteed to have similar condition numbers, provided $\min\{\|A_0\|_2,\|A_d\|_2\}$ is not too small. 
The goal of the following two examples is to illustrate this fact, and to show the benefits of scaling the polynomial. 

In the first experiment, we consider again the ``plasma drift'' matrix polynomial from the NLEVP collection \cite{NLEVP}, the Frobenius companion form $\mathcal{L}_1(\lambda)$ in \eqref{eq:exp2 Frobenius} and the block Kronecker companion forms $\mathcal{L}_2(\lambda)$, $\mathcal{L}_3(\lambda)$ and $\mathcal{L}_4(\lambda)$ in \eqref{eq:exp1 linearizations}.
In Figure \ref{fig:plasma_drift2}, we plot the ratios
\begin{equation}\label{eq:exp3 ratio}
\frac{\mathrm{coeff\,cond}_{\mathcal{L}_i}(\lambda_0)}{\mathrm{coeff\,cond}_P(\lambda_0)} \quad \mbox{for i}=1,2,3,4,
\end{equation}
for the scaled polynomial (lower figure) and the unscaled polynomial (upper figure).
The unscaled matrix polynomial $\rho(P)$ factor \eqref{eq:rho factor} of  order $10^8$,  which explains the large ratios in Figure \ref{fig:plasma_drift2}.
Notice how scaling brings a considerable improvement in the conditioning of the four block Kronecker linearizations. 
This improvement is predicted by our theory, since the scaled polynomial has a $\rho(P)$ factor \eqref{eq:rho factor} approximately equal to 100.

\begin{figure}[h]
\centering
\includegraphics[height=6.5cm, width=13cm]{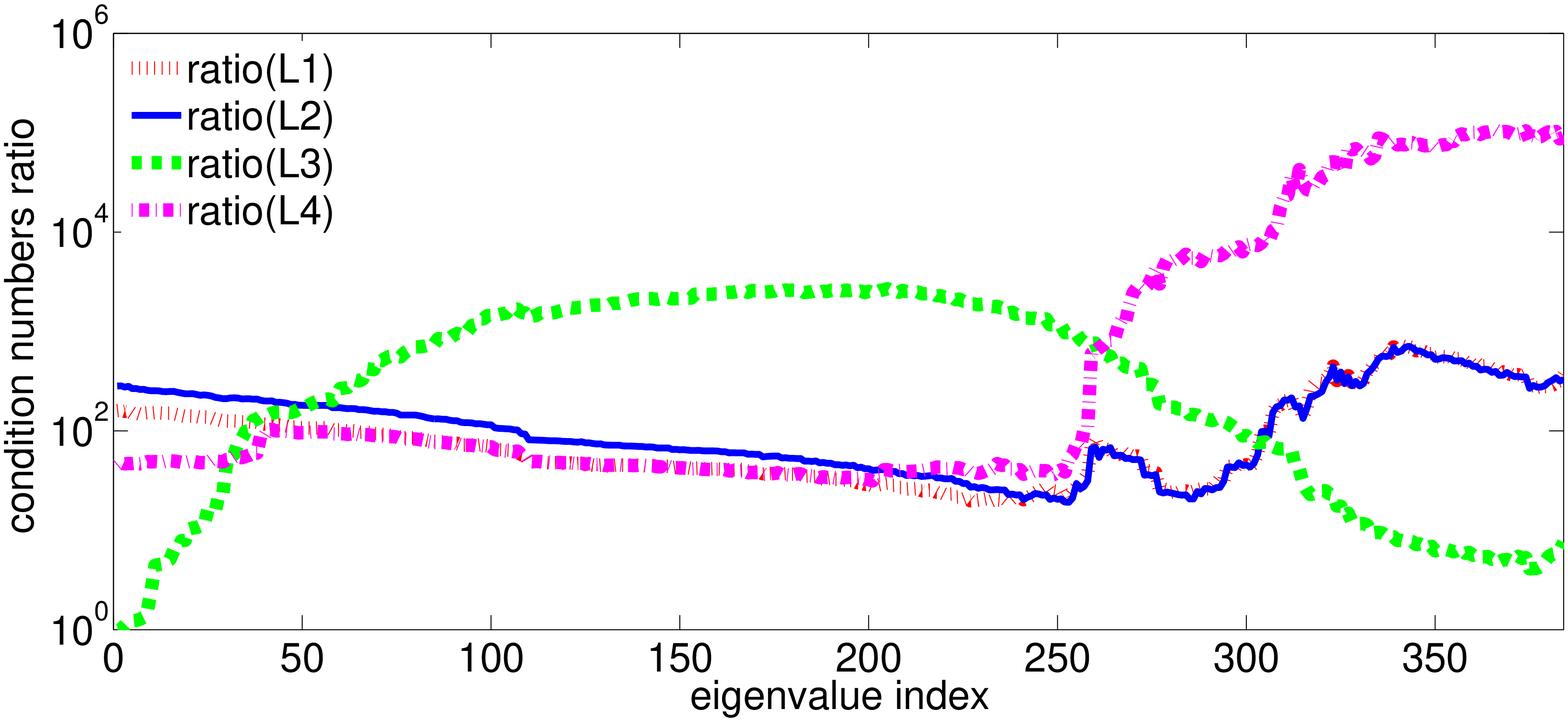}\\
\includegraphics[height=6.5cm, width=13cm]{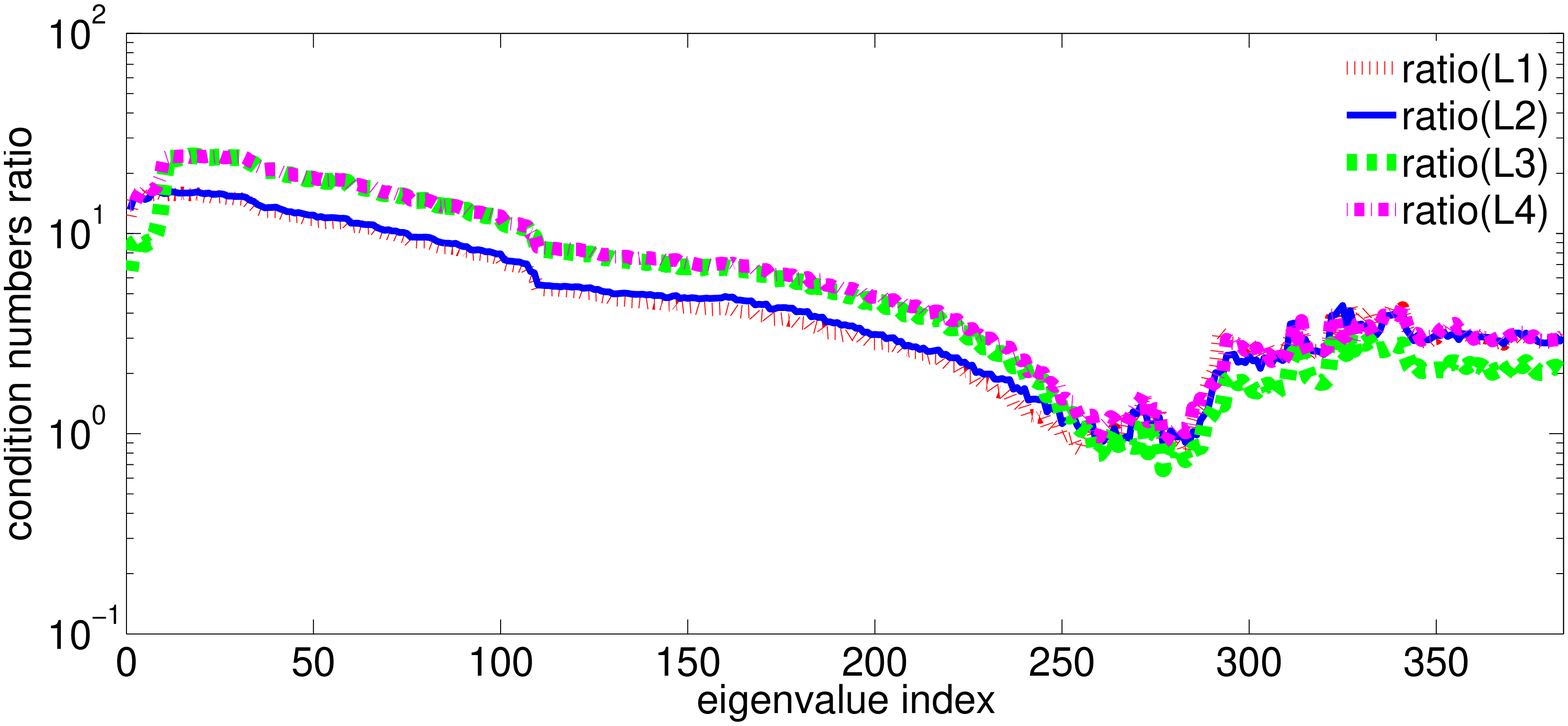}
\caption{Condition numbers ratios \eqref{eq:exp3 ratio} for the unscaled (upper figure) and scaled (lower figure) ``plasma drift'' matrix polynomial.
The pencils  $\mathcal{L}_1(\lambda)$, $\mathcal{L}_2(\lambda)$, $\mathcal{L}_3(\lambda)$ and $\mathcal{L}_4(\lambda)$ are as in \eqref{eq:exp2 Frobenius}-\eqref{eq:exp1 linearizations}.} 
\label{fig:plasma_drift2} 
\end{figure}

In the second experiment, we consider again the ``Orr-Sommerfeld'' matrix polynomial from the NLEVP collection \cite{NLEVP}, the Frobenius companion linearization $\mathcal{L}(\lambda)$ in \eqref{eq:exp2-L}, and the Frobeius-like quadratification $\mathcal{Q}(\lambda)$ in \eqref{eq:exp2-Q}. 
In Figure \ref{fig:orr2}, we plot the ratios
\begin{equation}\label{eq:exp3-ratio2}
\frac{\mathrm{coeff\,cond}_{\mathcal{L}}(\lambda_0)}{\mathrm{coeff\,cond}_P(\lambda_0)} \quad \quad \mbox{and} \quad \quad 
\frac{\mathrm{coeff\,cond}_{\mathcal{Q}}(\lambda_0)}{\mathrm{coeff\,cond}_P(\lambda_0)},
\end{equation}
for the scaled polynomial (lower figure) and the unscaled polynomial (upper figure).
We observe that both the scaled and unscaled matrix polynomials have very large factors \eqref{eq:rho factor}, which explains the large ratios in Figure \ref{fig:orr2}.
However, notice the significant improvement  that scaling the polynomial brings on the conditioning of the Frobenius companion form.
\begin{figure}[h]
\centering
\includegraphics[height=6.5cm, width=13cm]{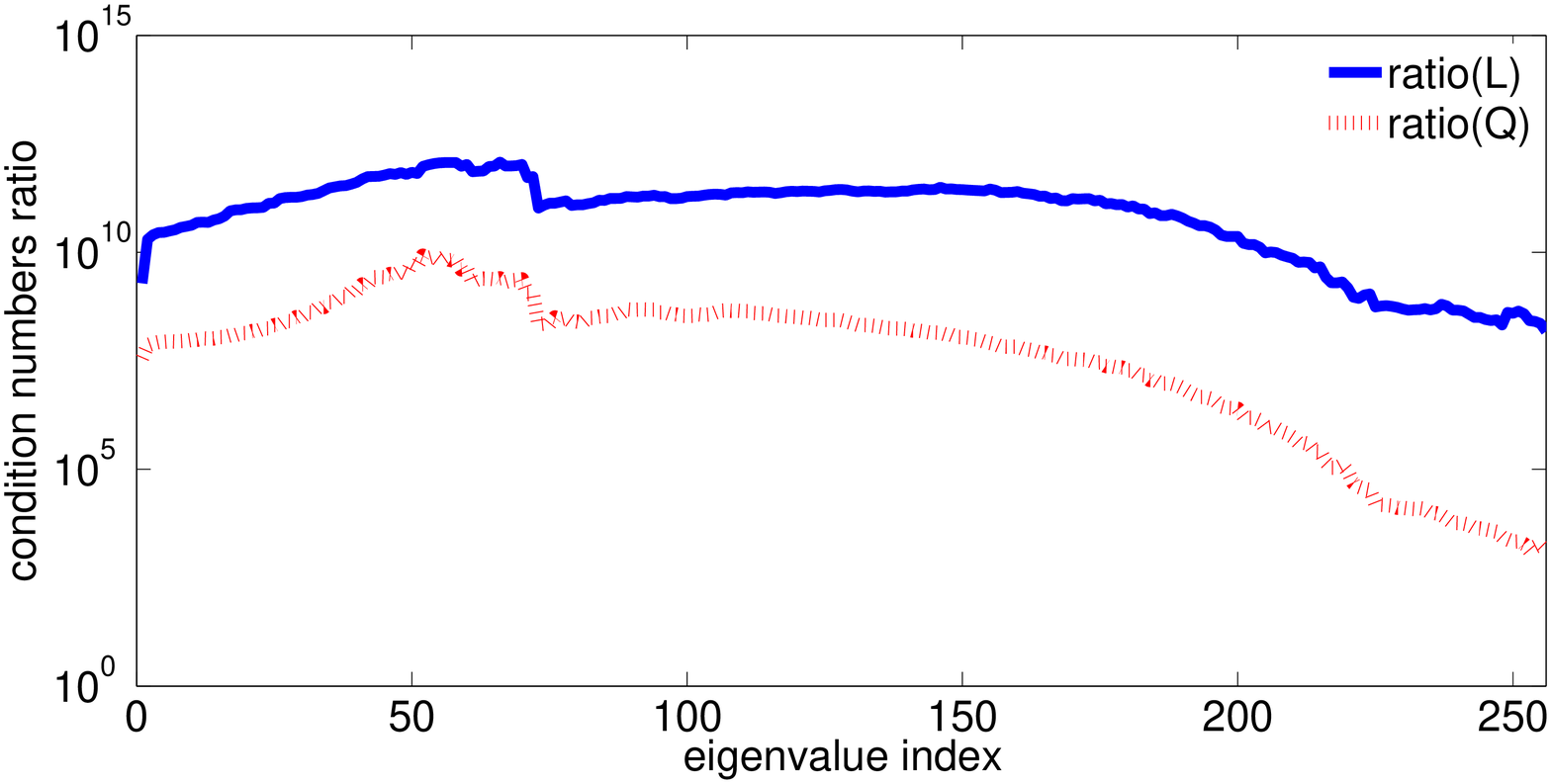}\\
\includegraphics[height=6.5cm, width=13cm]{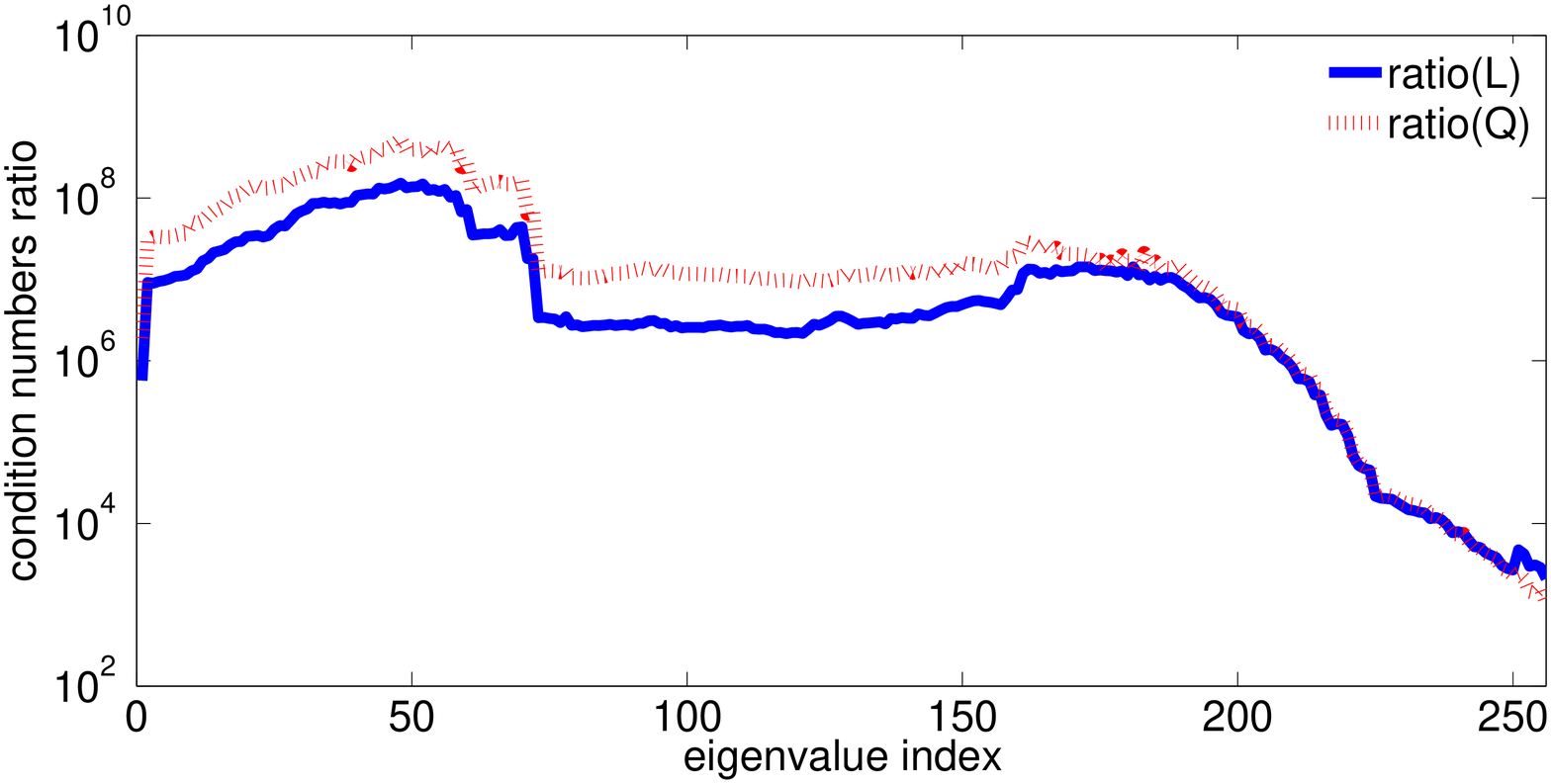}
\caption{Condition numbers ratio \eqref{eq:exp3-ratio2} for the ``Orr-Sommerfeld'' matrix polynomial.
Results for the unscaled polynomial are in the upper figure, and results for the scaled polynomial are in the lower figure.
The matrix polynomials $\mathcal{Q}(\lambda)$ and $\mathcal{L}(\lambda)$ are as in \eqref{eq:exp2-Q}-\eqref{eq:exp2-L}.} 
\label{fig:orr2} 
\end{figure}

\section{Conclusions}

Several recent papers have systematically addressed the task of broadening the menu of available $\ell$-ifcations \cite{Robol,DTDM10,FFP,ell-ifications-DPV,4m-vspace}.
Unfortunately, this explosion of new classes of $\ell$-ifications has not been followed by the corresponding analyses of the influence of the $\ell$-ification process on the accuracy and stability of the computed eigenvalues and/or eigenvectors.
Only  the influence of some classes of linearizations (the Frobenius companion forms and a block tridiagonal linearization \cite{tridiagonal}, linearizations in the $\mathbb{DL}(P)$ vector space \cite{DTDM10,conditioning,BackErrors,GoodVibrations,bivariate}), and Fiedler matrices \cite{Fiedler1,Fiedler2}) has been  studied in the last years.
In this work, we have started a systematic study of the numerical influence of $\ell$-ifications.
Focusing on the recent family of block Kronecker companion forms \cite{DLPV}, we have analyzed the influence of $\ell$-ifications on the conditioning of the polynomial eigenvalue problem.
Our findings lead to two main conclusions.
First, block Kronecker companion forms are  about as well conditioned as the polynomial itself, provided we scale $P(\lambda)$ so that $\max_{i=0:d}\{\|A_i\|_2\}=1$, and the quantity $\min\{\|A_0\|_2,\|A_d\|_2\}$ is not too small. 
Second, under the scaling assumption  $\max_{i=0:d}\{\|A_i\|_2\}=1$,  any block Kronecker companion form, regardless of its degree or block structure, is about as well-conditioned as the Frobenius companion forms.
We hove that the theoretical findings of this work will help to gain confidence on companion forms other than the Frobenius companion forms, and that this will lead to new algorithmic developments.


\end{document}